\documentclass{article}
\usepackage[margin = 1in]{geometry}

\usepackage{amsmath,amsthm,amssymb}
\usepackage{mathtools}
\usepackage{mathrsfs, eufrak}
\usepackage{stmaryrd}
\usepackage{enumitem}
\usepackage{hyperref}
\usepackage{multicol}
\allowdisplaybreaks

\usepackage{tikz-cd}
\usepackage{caption}
\usepackage{subcaption}
\usetikzlibrary{decorations.markings,arrows,quotes,angles,positioning,calc}
\tikzset{My style/.style={circle, minimum size=1pt, inner sep=1pt, fill}}

\newcommand{\C}{{\mathbb{C}}}
\newcommand{\R}{{\mathbb{R}}}
\newcommand{\Z}{{\mathbb{Z}}}
\newcommand{\N}{{\mathbb{N}}}

\renewcommand{\P}{{\mathcal{P}}}
\newcommand{\Pw}{{\P_w}}
\newcommand{\Gr}{{\mathcal{G}r}}

\newcommand{\F}{{\mathcal{F}}}

\newcommand{\h}{{\mathfrak{h}}}
\renewcommand{\t}{{\mathfrak{t}}}
\newcommand{\g}{{\mathfrak{g}}}
\renewcommand{\b}{{\mathfrak{b}}}
\newcommand{\n}{{\mathfrak{n}}}

\newcommand{\ui}{{\underline{i}}}
\newcommand{\uj}{{\underline{j}}}

\newcommand{\hs}{{\bar{s}}}
\newcommand{\tf}{{\tilde{f}}}
\newcommand{\te}{{\tilde{e}}}
\newcommand{\ep}{{\varepsilon}}
\newcommand{\wt}{{\operatorname{wt}}}
\newcommand{\new}{{\mathrm{new}}}
\newcommand{\trop}{{\mathrm{trop}}}

\newcommand{\conv}{{\operatorname{conv}}}
\newcommand{\Pol}{{\operatorname{Pol}}}
\newcommand{\dimvec}{{\underline{\operatorname{dim}}}}
\newcommand{\Ext}{{\operatorname{Ext}}}
\newcommand{\Hom}{{\operatorname{Hom}}}

\newtheorem{theorem}{Theorem}[section]
\newtheorem{lemma}[theorem]{Lemma}
\newtheorem{proposition}[theorem]{Proposition}
\newtheorem{corollary}[theorem]{Corollary}
\newtheorem{conjecture}[theorem]{Conjecture}
\newtheorem{theoremALPH}{Theorem}

\newtheorem{conjectureALPH}[theoremALPH]{Conjecture}

\theoremstyle{definition}
\newtheorem{definition}[theorem]{Definition}
\newtheorem{question}[theorem]{Question}
\newtheorem{example}[theorem]{Example}
\newtheorem{remark}[theorem]{Remark}

\newtheorem{claim}{Claim}
\newenvironment{claimproof}[1]{\par\noindent\underline{Proof:}\space#1}{\hfill $\blacksquare$}

\author{Kathlyn Dykes}
\title{MV polytopes and reduced double Bruhat cells}
\date{}

\begin{document}

\maketitle

\begin{abstract}
When $G$ is a complex reductive algebraic group, MV polytopes are in bijection with the non-negative tropical points of the unipotent group of $G$. By fixing $w$ from the Weyl group, we can define MV polytopes whose highest vertex is labelled by $w$. We show that these polytopes are in bijection with the non-negative tropical points of the reduced double Bruhat cell labelled by $w^{-1}$. To do this, we define a collection of generalized minor functions $\Delta_\gamma^\new$ which tropicalize on the reduced Bruhat cell to the BZ data of an MV polytope of highest vertex $w$. 

We also describe the combinatorial structure of MV polytopes of highest vertex $w$. We explicitly describe the map from the Weyl group to the subset of elements bounded by $w$ in the Bruhat order which sends $u \mapsto v$ if the vertex labelled by $u$ coincides with the vertex labelled by $v$ for every MV polytope of highest vertex $w$. As a consequence of this map, we prove that these polytopes have vertices labelled by Weyl group elements less than $w$ in the Bruhat order.
\end{abstract}

\tableofcontents

\section{Introduction}

For $G$ a complex reductive algebraic group, the irreducible representations are highest weight representations. To understand the tensor products of these irreducible representations, Lusztig defined a canonical basis for each $V(\omega_i)$, which behaves nicely with the decomposition of these tensor products into their irreducible subrepresentations \cite{CanonicalBases}. In \cite{MirkovicVilonen2007}, Mirkov\'ic and Vilonen provide another basis using the geometric Satake correspondence, which relates the representation theory of the Langlands dual group $G^\vee$ with the intersection homology of the affine Grassmanian, $\Gr$.  

Under this correspondence, the bases of the representations correspond to certain subvarieties of $\Gr$, called Mirkov\'ic-Vilonen (MV) cycles. These MV cycles are the irreducible components of the intersection of infinite cells and as such, are difficult to understand as geometric objects. Anderson first conjectured that MV cycles could be analysed by studying their moment polytopes \cite{Anderson} and in \cite{MVpolytopes}, Kamnitzer gives a combinatorial description of MV cycles using these moment polytopes, called MV polytopes. Goncharov and Shen \cite{Geometryofcanonical} take this one step further by explicitly showing that the set of MV polytopes are the tropical points of the unipotent subgroup of $G$. The benefit to this point of view is that the tropical Pl\"{u}cker relations come from the Pl\"{u}cker relations on $N$, which arise naturally by studying the transition maps of Lusztig's positive atlas \cite{TotalPositivityin}. 

In this paper, for $w \in W$, we investigate a subset of MV polytopes called \emph{MV polytopes of highest vertex $w$}, denoted by $\Pw$. These polytopes are MV polytopes whose vertex labelled by $w$ is equal to the vertex labelled by $w_0$.

The original motivation to study these polytopes was to develop a better understanding of affine MV polytopes, although these MV polytopes are also of interest due to their connection to preprojective algebra modules and MV cycles. In \cite{AffineMVpolytopes}, the authors define a class of preprojective algebra modules of interest, $\mathscr{T}^w$ and in \cite{Menard:Richardsonvarieties}, M\'{e}nard proves that $\Pw$ is exactly set of MV polytopes associated to these modules. 

For any MV polytope, there is a canonical labelling of the vertices by the Weyl group, so that the vertex data of $P\in \Pw$ can be labelled $(\mu_v)_{v\in W}$. The main result of this paper is that the vertex data is only dependent on the Weyl group elements bounded by $w$. 
\begin{theoremALPH}[Theorem \ref{theorem:mu_w}, Corollary \ref{corollary:vertexdata}]\label{theoremA}
For every $P \in \Pw$ with vertex data $(\mu_v)_{v\in W}$, $P = \text{conv} \{ \mu_v: v \in W, v \leq w\}$. 
\end{theoremALPH}
This theorem is proven by explicitly describing the map $W \rightarrow \{ v \leq w\}$. 

We would also like to realize $\Pw$ as the non-negative tropical points on some subvariety of $N$ such that the tropicalized generalized minors functions send a non-negative tropical point to the BZ data of an MV polytope of highest vertex $w$. The candidate for this subvariety is the reduced double Bruhat cell, $L^{w^{-1}} = N \cap B_- w^{-1} B_-$. On this subvariety, some of these generalized minor functions vanish. Instead, we redefine these minors $\Delta_{v\omega_i}^\new$ to be the smallest weight $\gamma$ such that $\Delta_{\gamma, v\omega_i}\neq 0$ on $L^{w^{-1}}$. Consider the collection of tropical functions $M_\gamma = (\Delta_\gamma^\new \circ \eta_{w^{-1}}^{-1})^\trop$ for $\gamma \in \Gamma$, where $\eta_{w^{-1}}$ is a necessary change of coordinates. 

\begin{theoremALPH}[Proposition \ref{proposition:edgeequalities}]\label{theoremB}
On $L^{w^{-1}}(\Z^\trop)_\geq$, the collection $(M_{\gamma})_{\gamma \in \Gamma}$ satisfies the following conditions:
\begin{enumerate}[label=(\roman*)]
\item the edge inequalities,
\item the tropical Pl\"{u}cker relations on the subcollection $(M_{\gamma})_{\gamma \in \Gamma^w}$, 
\end{enumerate}
\end{theoremALPH}

\begin{conjectureALPH}
On $L^{w^{-1}}(\Z^\trop)_\geq$, the collection $(M_{\gamma})_{\gamma \in \Gamma}$ satisfy the edge equalities on the subcollection $(M_\gamma)_{\gamma \in \Gamma \setminus \Gamma^w}$.
\end{conjectureALPH}

Using this new collection of tropical functions $(M_\gamma)_{\gamma \in \Gamma}$, we obtain an identical result to the case of $N$. 

\begin{theoremALPH}[Theorem \ref{theorem:BZdata}]
There is a bijection $L^{w^{-1}}(\Z^\trop)_{\geq} \longrightarrow \Pw$ by  $\ell \rightarrow (M_\gamma(\ell))_{\gamma \in \Gamma}$. 
\end{theoremALPH}

This paper is organized as follows. In Section \ref{section:MVpolytopes}, we give a brief background of the theory of MV polytopes. In Section \ref{section:combinatorialdata}, we describe the Lusztig and vertex data of $\Pw$ and prove Theorem \ref{theoremA}. In Section \ref{section:tropicalpoints}, we outline the theory which relation MV polytopes of highest vertex $w$ to the tropical points of the reduced double Bruhat cells. 

\subsection*{Acknowledgments}

First, I thank my supervisor, Joel Kamnitzer, for his guidance and suggestions during this project. 
I would also like to thank Jiuzu Hong, Florian Herzig, Marco Gualtieri, Lisa Jeffrey and Peter Tingley for their suggestions and corrections. I thank Pierre Baumann for allowing me to include his proof of Lemma \ref{lemma:generaldiagonalrelations-simplylaced}. 

During this work, I was supported by an NSERC graduate scholarship. 

\section{Notation}

Let $G$ be a semisimple, simply connected, complex group. Let $T$ be a maximal torus of $G$. We define the weight and coweight lattice as $X^* = \text{Hom}(T, \C^\times)$ and $X_* = \text{Hom}(T, \C^\times)$ respectively. Let $W = N_G(T)/T$ be the Weyl group. 

Fix $B$ be a Borel subgroup of $G$ such that $T \subset B$. Let $N$ be the unipotent subgroup of $B$. Let $I$ be the index set of the simple roots and denote $\alpha_i$ as the simple root associated to the index $i$ while $\alpha_i^\vee$ is the simple coroot. Let $\Delta$ be the set of roots and $\Delta_+$ the set of positive roots while $\Delta^\vee$ is set of coroots and $\Delta_+^\vee$ the set of positive coroots. Let $\langle \cdot, \cdot \rangle: X_* \times X^* \rightarrow \C$ be the pairing of the weight and the coweight lattice and set $a_{ij} = \langle \alpha_i^\vee, \alpha_j \rangle$.  Denote by $Q = \N \Delta$  the root lattice so $Q_+ = \N \Delta_+$ is the positive root cone. Similarly, let $Q^\vee = \N\Delta^\vee$  be the coroot lattice and $Q_+ = \N\Delta_+^\vee$ the positive coroot cone. Let $\omega_i$ be the fundamental weights, which form a basis of the weight lattice $X^*$ such that $\langle \alpha_i^\vee, \omega_j \rangle = \delta_{i,j}$. 

Consider the space $\t_\R = X_*\otimes \R$ and $\t^*_\R = X^*\otimes \R$. Define a partial order on $X_*$ by $\mu \leq \lambda \iff \lambda - \mu \in Q^\vee_+$ and a partial order on $X^*$ by $\mu \leq \lambda \iff \lambda - \mu \in Q_+$. Define the twisted partial order $\geq_w$ on $\t_\R$ by $\beta \leq_w \alpha \iff \langle \beta - \alpha, w \omega_i \rangle \geq 0$ for all $i \in I$.

Let $s_i$ be the simple reflection associated to the simple root $\alpha_i$, i.e. $s_i(\alpha) = \alpha - \langle \alpha_i^\vee, \alpha \rangle \alpha_i$ and set $S = \{s_i: i \in I\}$. Then $W$ is also the Coxeter group generated by $S$. $W$ acts on the weight lattice by $s_i(\beta) = \beta - \langle \alpha_i^\vee, \beta \rangle \, \alpha_i$ for $\beta \in X^*$. Similarly, $W$ acts on the coroots and the coweight lattice by $s_i(\beta) = \beta - \langle \beta , \alpha_i \rangle \, \alpha_i^\vee$ for $\beta \in X_*$. 

For $w \in W$, let $\ell(w)$ denote the length of $w$. We say the product $s_{i_1} \cdots s_{i_k}$ is \emph{reduced} if $k = \ell(w)$. The tuple of indices $\ui = (i_1, \dots, i_k)$ a \emph{word} of $w$ if $w = s_{i_1} \cdots s_{i_k}$. A \emph{reduced word} of $w$ is $\ui$ such that $w = s_{i_1} \cdots s_{i_k}$ is reduced. 

Let $\leq$ denote the Bruhat order and let $\leq_R, \leq_L$ denote the right and left weak Bruhat orders respectively. We will denote intervals in the strong Bruhat order by $[v,w] = \{ x: v\leq x\leq w\}$. Similarly,  the weak Bruhat intervals are $[v,w]_R = \{ x: v \leq_R x \leq_R w\}$ and $[v,w]_L = \{ x: v \leq_L x \leq_L w\}$.

\section{MV polytopes}\label{section:MVpolytopes}

MV polytopes were originally defined by Anderson \cite{Anderson} as the moment polytopes of certain subvarieties of the affine Grassmanian called MV cycles. In \cite{MVpolytopes}, Kamnitzer gave a completely combinatorial description of MV polytopes using their hyperplane data. In particular, a GGMS polytope is an MV polytope exactly when the hyperplane data are a BZ datum. In this section, we review MV polytopes as combinatorial objects and outline their relation to preprojective algebra modules. We describe the crystal structure on the set of MV polytopes and define the Saito crystal reflection.

To define MV polytopes, we first consider GGMS polytopes. 

\begin{definition}
Consider a collection $\mu_\bullet = (\mu_w)_{w\in W }$  in the coroot lattice $Q^\vee$ such that $\mu_{v} \leq_w \mu_w$ for all $v, w \in W$. A Gelfand-Goresky-MacPherson-Serganova (GGMS) polytope is a convex polytope $P(\mu_\bullet)$ of the form $P(\mu_\bullet)=\bigcap_{w \in W} C_w^{\mu_w}$ where
\[
C_w^{\mu_w} = \lbrace x \in \t_\R : \langle x, w \cdot \omega_i \rangle \geq \langle \mu_w, w \cdot \omega_i \rangle, \forall i\rbrace.
\]

\end{definition}
By {\cite[Proposition 2.2]{MVpolytopes}}, $P(\mu_\bullet) = \conv \{ \mu_w: w \in W\}$. We call $(\mu_\bullet)$ the \emph{vertex data} of the polytope. 

We can also define a GGMS polytope using the hyperplane data. The hyperplanes are indexed by weights of the form $w \omega_i$. Define the set of \emph{chamber weights} $\Gamma = \{ w\omega_i: w\in W, i \in I\}$.  Let $M_\bullet = (M_\gamma)_{\gamma \in \Gamma}$ be a collection of integers that satisfy the \emph{edge inequalities} for each $w \in W$ and $i \in I$:
\begin{align}\label{equation:defofedgeequalities}
M_{ws_i\omega_i} + M_{w\omega_i} + \sum_{j \neq i} a_{ji} M_{w\omega_j} \leq 0
\end{align} 
where $a_{ji} = \langle \alpha_j^\vee, \alpha_i \rangle$. Then the polytope $P(M_\bullet)$ defined by the hyperplane data is
\[
P(M_\bullet) = \{ x \in \t_\R: \langle x, \gamma \rangle \geq M_\gamma, \forall \gamma \in \Gamma\}.
\]

By {\cite[Proposition 2.2]{MVpolytopes}}, these two definitions are equivalent in the following way. If $P = P(\mu_\bullet)$, then $P= P(M_\bullet)$ where we set $M_{w\omega_i} = \langle \mu_w, w \cdot \omega_i \rangle$. If $P = P(M_\bullet)$ then $P = P(\mu_\bullet)$ where we set $\mu_w = \sum_{i\in I} M_{w\omega_i} w \cdot \alpha_i^\vee$. From now on, for a GGMS polytope $P$, we will denote $(\mu_\bullet)$ as the vertex data and $(M_\bullet)$ as the hyperplane data. 

For $w \in W$ and $s_i$ such that $\ell(s_iw) > \ell(w)$, there is an edge in $P(\mu_\bullet)$ connecting $\mu_w$ and $\mu_{ws_i}$ where
\begin{align}\label{equation:lusztigdatafromvertexdata}
\mu_{ws_i} - \mu_{w} = c w \cdot \alpha_i^\vee
\end{align}
and $c = -M_{w\omega_i} - M_{ws_i\omega_i} - \sum_{j \neq i} a_{ji} M_{w\omega_j}$. Note that from the edge inequalities (\ref{equation:defofedgeequalities}) $c \geq 0$. We call $c$ the \emph{length} of the edge from $\mu_w$ to $\mu_{ws_i}$. 
The next lemma follows directly from (\ref{equation:lusztigdatafromvertexdata}). 

\begin{lemma}\label{lemma:edgesandLusztig} 
For $P$ a GGMS polytope with vertex data $(\mu_\bullet)$ and hyperplane data $(M_\bullet)$, $\mu_{ws_i} - \mu_{w} =0 \iff M_{w \omega_i} +  M_{w s_i \omega_i} = -\sum_{j \neq i} a_{ji} M_{w\omega_j}$.
\end{lemma}

\begin{example}
For $G = SL_3$, the simple coroots are given by $\alpha^\vee_1 = (1,-1,0)$, $\alpha^\vee_2 = (0, 1,-1)$  so these GGMS polytopes are actually polygons. See Figure \ref{figure:A2polytopebackground} for an example. 

The fundamental weights are $\omega_1 = (1, 0,0)$, $\omega_2 = (1,1,0)$ and the chamber weights are 
\[
\Gamma = \{ \omega_1, \omega_2, s_1\omega_1, s_2 \omega_2, s_2s_1\omega_1, s_1s_2\omega_2\}.
\]
 These chamber weights index the hyperplanes $(M_\bullet)$ as in Figure \ref{figure:A2polytopebackground}. 

\begin{figure}[h]
\centering
$
\begin{tikzpicture}
\node[My style] (S1) at (0.965925826, 0.258819045)[label=right :$\mu_{s_1}$] {};
\node[My style] (S12) at (1.48357706606, 2.190621620632)[label=right:$\mu_{s_1s_2}$] {};
\node[My style] (SE) at (0,0)[label=below:$\mu_{e}$] {};
\node[My style] (W0) at (0.0693766541199999, 3.604786106264)[label=above:$\mu_{w_0}$] {};
\node[My style] (S2) at (-2.12130061791, 2.121246728448)[label=left:$\mu_{s_2}$] {};
\node[My style] (S21) at (-1.86247499788, 3.087148016264)[label=left:$\mu_{s_2s_1}$] {};

\draw[thick] (SE) -- (S1) -- (S12) -- (W0) -- (S21) -- (S2) -- (SE);

\draw[line width=1pt,-stealth](0.482962913, 0.1294095225)--(0.3535533905, 0.6123724355) node[anchor=south]{\footnotesize{$\omega_2$}};

\draw[line width=1pt,-stealth](-1.060650308955, 1.060623364224)--(-0.7196619112295, 1.46102086977)
 node[anchor=south]{\footnotesize{$\omega_1$}};

\draw[line width=1pt,-stealth](-1.991887807895, 2.604197372356)
--(-1.54830184399870, 2.48144311823562) node[anchor=north]{\footnotesize{$s_2\omega_2$}};

\draw[line width=1pt,-stealth](1.22889255075, 1.2092655196)--(0.7853065868537, 1.33201977372038)
 node[anchor=south]{\footnotesize{$s_1\omega_1$}};

\draw[line width=1pt,-stealth](0.78061796481, 2.882249050232) --(0.4396295670845, 2.481851544686) node[anchor=north]{\footnotesize{$s_1s_2\omega_2$}};

\draw[line width=1pt,-stealth](-0.89654917188, 3.345967061264)--(-0.76713964938, 2.863004148264) node[anchor=north]{\footnotesize{$s_2s_1\omega_1$}};

\end{tikzpicture}
$
\caption{An $A_2$ MV polytope}
\label{figure:A2polytopebackground}
\end{figure}
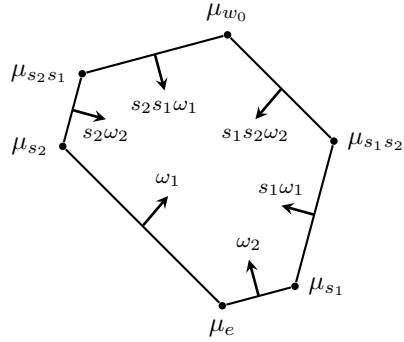
\end{example} 

When a GGMS polytope is an MV polytope, the hyperplane data satisfy certain relations. First, we recall the tropical Pl\"{u}cker relations, which come from the tropicalization of the Pl\"{u}cker relations of \cite{TotalPositivityin}. 

\begin{definition}
The collection $(M_\gamma)_{\gamma \in \Gamma}$ satisfies the tropical Pl\"{u}cker relations if for each $w\in W$ and every $i, j \in I$ such that $i\neq j$ and $s_i,s_j \not\in D_R(w)$, then either $a_{ij}=0$ or the following holds:
\begin{enumerate}[label=\arabic*)]
 \item If $a_{ij} = a_{ji} =-1$, then
 \[
 M_{ws_i\omega_i} + M_{ws_j\omega_j} = \min \{ M_{w\omega_i} + M_{ws_is_j\omega_j},  M_{ws_js_i\omega_i} +M_{w\omega_j}\}
 \]
 \item If $a_{ij} = -1, a_{ji}=-2$, then
 \begin{align*}
 M_{ws_j\omega_j} + M_{ws_is_j\omega_j} + M_{ws_i\omega_i} = \min\{ & 2M_{ws_is_j\omega_j} + M_{w\omega_i},2 M_{w\omega_j} + M_{ws_is_js_i\omega_i}, \\
 & M_{\omega_j} + M_{ws_js_is_j\omega_j} + M_{ws_i\omega_i} \}\\
 M_{ws_js_i\omega_i} + 2M_{ws_is_j\omega_j} + M_{ws_i\omega_i} = \min \{ &2M_{w\omega_j} + 2M_{ws_is_js_i\omega_i}, 2M_{ws_js_is_j\omega_j} + M_{ws_i\omega_i},\\
 & M_{ws_is_js_i\omega_i} + 2M_{ws_is_j\omega_j} + M_{w\omega_i} \}
 \end{align*}
 \item If $a_{ij} = -2, a_{ji} =-1$, then
 \begin{align*}
 M_{ws_js_i\omega_i} + M_{ws_i\omega_i} + M_{ws_is_j\omega_j}  = \min \{ &2M_{ws_i\omega_i} + M_{ws_js_is_j\omega_j}, 2M_{ws_is_js_i\omega_i} + M_{w\omega_j}, \\
 &M_{ws_is_js_i\omega_i} + M_{w\omega_i} + M_{ws_is_j\omega_j}\}\\
 M_{ws_j\omega_j} + 2M_{ws_i\omega_i} + M_{ws_is_j\omega_j} = \min \{ &2M_{ws_is_js_i\omega_i} + 2M_{w\omega_j}, 2M_{w\omega_i} + 2M_{ws_is_j\omega_j}, \\
 &M_{w\omega_j} + 2M_{ws_i\omega_i} + M_{ws_js_is_j\omega_j} \}
 \end{align*}
\end{enumerate}
If $a_{ij} = -3$ or $a_{ji} =-3$, the tropical Pl\"{u}cker relations are given in \cite[Proposition 4.2]{TotalPositivityin}. We omit them here due to length. 
\end{definition}

Note that the tropical Pl\"{u}cker relations impose conditions on each $2$-face of $P$. 

\begin{definition}\label{definition:BZdata}
The collection $(M_\gamma)_{\gamma \in \Gamma}$ is a Berenstein-Zelevinsky (BZ) datum of coweight $\lambda$ if:
\begin{enumerate}[label=(\roman*)]
 \item $(M_\bullet)$ satisfies the tropical Pl\"{u}cker relations,
 \item $(M_\bullet)$ satisfies the edge inequalities (\ref{equation:defofedgeequalities}),
 \item $M_{\omega_i} = 0$ and $M_{w_0 \cdot \omega_i} = \langle \lambda, w_0 \cdot \omega_i \rangle$.
\end{enumerate}
\end{definition}

We define an MV polytope as GGMS polytope $P$ whose hyperplane data $(M_\bullet)$ are a BZ datum. This definition is equivalent to the original definition of MV polytopes as the moment polytopes of MV cycles.

\begin{theorem}[{\cite[Theorem 3.1]{MVpolytopes}}]
A GGMS polytope $P(M_\bullet)$ is an MV polytope if and only if it is the moment polytope of a stable MV cycle. 
\end{theorem}

Denote by $\P$ the set of MV polytopes. For any $P \in \P$, the polytope is determined by its vertex data $(\mu_\bullet)$, which are a collection of points in $Q^\vee$, or its BZ data $(M_\bullet)$, which are a collection of integers. There is one more set of combinatorial data which determines $P$, closely related to the vertex data. 

For a reduced word $\ui = (i_1, \dots, i_m)$ of $w_0$,  define the Weyl group elements $w_k^\ui = s_{i_1}\cdots s_{i_k}$ for $1 \leq k \leq m$ and set $w_0^\ui =e$. The reduced word $\ui$ gives a path $\mu_e$, $\mu_{w_1^{\ui}}$, $\mu_{w_2^{\ui}}$, $\dots$, $\mu_{w_{m-1}^{\ui}}$, $\mu_{w_0}$ in the 1-skeleton of $P$.  From (\ref{equation:lusztigdatafromvertexdata}), 
\[
\mu_{w_k^\ui} - \mu_{w_{k-1}^\ui} = \left(-M_{w_{k-1}^\ui \omega_{i_k}} - M_{w_k^\ui\omega_{i_k}} - \sum_{j \neq {i_k}} a_{ji_k} M_{w_{k-1}^\ui\omega_j}\right) w_{k-1}^\ui\alpha_{i_k}^\vee
\]

\begin{definition}
Let $P\in \P$ with vertex data $(\mu_\bullet)$ and BZ data $(M_\bullet)$. For a reduced word $\ui = (i_1, \dots, i_m)$ of $w_0$, the \emph{Lusztig data of $P$ with respect to $\ui$} is defined by $n_k^\ui = -M_{w_{k-1}^\ui \omega_{i_k}} - M_{w_k^\ui\omega_{i_k}} - \sum_{j \neq {i_k}} a_{ji_k} M_{w_{k-1}^\ui\omega_j}$. 
\end{definition}

By the edge inequalities (\ref{equation:defofedgeequalities}), $n_k^\ui \geq 0$ for $1 \leq k \leq m$. The Lusztig data corresponds to the lengths of the edges along the path determined by $\ui$ above. Note that for any $P \in \P$ and any $\ui$, $n^\ui_\bullet(P) \in \N^m$. 

For convenience, we will call the path $\mu_e$, $\mu_{s_{i_1}}$, $\dots$, $\mu_{s_{i_1} \cdots s_{i_{m-1}}}$, $\mu_{w_0}$ determined by a reduced word $\ui$ of $w_0$ a \emph{minimal path from $\mu_e$ to $\mu_{w_0}$} in $P$. We will also use the shorthand $n_\bullet^\ui:= n_\bullet^{\ui}(P)$ when it is clear what $P$ is.  

\begin{example}
For the $A_2$ polytope in Figure \ref{figure:A2polytopebackground}, the reduced word $\ui = (1,2,1)$ gives the Lusztig data $n_\bullet^{121} = (1,2,2)$, which are the lengths of the edges on the right side of the polytope. For $\ui = (2,1,2)$, $n_\bullet^{212} = (3,1,2)$ which are the lengths of the edges on the left side of the polytope. 
\end{example}

Any MV polytope is completely determined by its Lusztig data along one minimal path. 

\begin{theorem}[{\cite[Theorem 7.1]{MVpolytopes}}]\label{theorem:MVpolytopeslusztigdata}
Let $\ui$ be any reduced word of $w_0$. The Lusztig data with respect to $\ui$ gives a bijection $\P \rightarrow \N^m$. 
\end{theorem}

\subsection{Crystal structure of $\P$}\label{section:crystals}

The set of MV polytopes has a bicrystal structure and hence a reflection of the crystal will result in an action on the set of MV polytopes. First, we define a crystal structure as in {\cite[Section 7.2]{OnCrystalBases}}. 

\begin{definition}
A \emph{crystal} is a set $B$ along with the maps 
\begin{align*}
\wt : B \rightarrow X_*, && \te_i: B \rightarrow B \sqcup \{0\} && \tf_i : B \rightarrow B \sqcup \{0\}, && \ep_i: B \rightarrow \Z \cup\{-\infty\}, && \varphi_i: B \rightarrow \Z \cup\{-\infty\}
\end{align*}
for each $i \in I$ with the following axioms:
\begin{enumerate}[label=\arabic*)]
 \item For all $b \in B, i \in I$, $\varphi_i(b) = \ep_i(b) + \langle \wt(b), \alpha_i \rangle$
 \item If $b \in B, i \in I$ and $\te_i (b) \neq 0$, then 
 \begin{align*}
 \wt(\te_i (b)) = \wt(b) + \alpha_i^\vee, && \ep_i(\te_i(b)) = \ep_i(b) -1, &&\varphi_i(e_i(b)) = \varphi_i(b)+1
 \end{align*}
 \item If $b \in B, i \in I$ and $\tf_i (b) \neq 0$, then 
 \begin{align*}
 \wt(\tf_i(b)) = \wt(b) - \alpha_i^\vee, && \ep_i(\tf_i(b)) = \ep_i(b) +1, &&\varphi_i(f_i(b)) = \varphi(b) -1
 \end{align*}
 \item $b' = \te_i(b) \iff f_i(b') = b$
\end{enumerate}
A \emph{highest weight crystal} has a unique element $b_0$ such that $b_0$ can be obtained by any element $b \in B$ by applying a sequence of $\te_i$ for different $i \in I$. 
\end{definition}

In particular, we are interested in the crystal $B(\infty)$. This is the highest weight crystal determined by the relations $\wt(b_0)=0$ and $\ep_i(b) =  \max \{n: \te_i^n b \neq 0\}$.

Let $*$ denote Kashiwara's involution on $B(\infty)$ \cite{KashiwaraInvolution}. Define $\te_i^* = * \circ \te_i \circ *$, $\tf_i^* = * \circ \tf_i \circ *$, $\ep_i^*(b) = \ep(*b)$ and $\varphi_i^* = \varphi_i(*b)$ for every $i \in I, b \in B(\infty)$. Then $(B(\infty), \wt, \ep_i^*, \varphi_i^*, \te_i^*, \tf_i^*)$ is also a crystal.  We call $B(\infty)$ a \emph{bicrystal} with these two crystal structures where the weight functions agree and $\wt(b) \in -Q_+$ for every $b \in B(\infty)$.  In \cite{MVcrystals}, Kamnitzer defines the bicrystal structure on the set of MV polytopes and proves that this structure is isomorphic to the $B(\infty)$ bicrystal. 
\begin{theorem}[{\cite[Theorem 6.2, Corollary 6.3]{MVcrystals}}]\label{theorem:MVcrystalstructure}
Let $P$ be an MV polytope with vertex data $(\mu_\bullet)$. 
\begin{enumerate}[label=\arabic*)]
 \item $\tf_j(P)$ is the unique MV polytope with vertex data $(\mu_\bullet')$ where
 \[
 \mu_e' = \mu_e \text{ and } \mu_{w}' = \mu_{w} + \alpha_j^\vee \text{ if } s_jw < w.
 \]
 \item $\te_j(P)=0 \iff \mu_e = \mu_{s_j}$. Otherwise, $\te_j(P)$ is the unique MV polytope with vertex data $(\mu'_\bullet)$ where
\[
\mu_e' = \mu_e \text{ and } \mu_w' = \mu_w - \alpha_j^\vee \text{ if } s_j w<w.
\]
 \item $\tf^*_j (P)$ is the unique MV polytope with vertex data $(\mu'_\bullet)$ where
 \[
 \mu'_{w_0} = \mu_{w_0} + \alpha_j^\vee \text{ and } \mu_w' = \mu_w \text{ for } s_jw>w.
 \]
 \item $\te^*_j (P)=0 \iff \mu_{w_0s_j} = \mu_{w_0}$. Otherwise, $\te^*_j(P)$ is the unique MV polytope with vertex data $(\mu'_\bullet)$ where
 \[
 \mu'_{w_0} = \mu_{w_0} - \alpha_j^\vee \text{ and } \mu_w' = \mu_w \text{ for } s_jw>w.
 \]
\end{enumerate}
\end{theorem}

\begin{example}\label{example:crystal}
When $G = SL_3$, consider the polytope $P$ given by the Lusztig data $n_\bullet^{(1,2,1)}(P) = (1,0,2)$ and $n_\bullet^{(2,1,2)}(P) = (1,1,0)$. Then the crystal operators act as follows:
\begin{align*}
n_\bullet^{(1,2,1)}(\tf_1(P)) &= (2,0,2), & n_\bullet^{(1,2,1)}(\te_1(P)) &= (0,0,2), \\ n_\bullet^{(1,2,1)}(\tf_2^*(P)) &= (1,0,3), & n_\bullet^{(1,2,1)} (\te_2^*(P)) &=(1,0,1) \\
n_\bullet^{(2,1,2)}(\tf_2(P)) &= (2,1,0), & n_\bullet^{(2,1,2)}(\te_2(P)) &= (0,1,0), \\ n_\bullet^{(2,1,2)}(\tf_1^*(P)) & = (1,1,1), & n_\bullet^{(2,1,2)}(\te_1^*(P)) & = 0.
\end{align*}
Note that $\tf_2(P) = \tf_2^*(P)$ and $\te_2(P) = \te_2^*(P)$. 
\end{example}

For each $j$, define $j^*$ to be the index such that $s_{j^*} w_0 = w_0 s_j$. Note that for the reduced word $\ui = (i_1, \dots, i_m)$ of $w_0$, $s_{i_1} \cdots s_{i_{m-1}} \alpha_{i_m} = \alpha_{i^*_m}$ so that $f_{i_m^*}^*$ and $e_{i_m^*}^*$ will change the last component of the Lusztig data with respect to $\ui$. 

We can explicitly see how these operators act on the Lusztig data of a polytope. Suppose $n_\bullet^\ui(P)$ is the Lusztig data with respect to $\ui$ for a polytope $P \in \P$. Then
\begin{align*}
n_\bullet^\ui(\tf_{i_1}(P)) &= (n_1 +1, n_2, \dots, n_m), & n_\bullet^\ui (\te_{i_1}(P)) &= (n_1 -1, n_2, \dots, n_m),\\
n_\bullet^\ui(\tf_{i^*_m}^*(P)) &= (n_1, \dots, n_{m-1}, n_m +1), & n_\bullet^\ui(\te_{i^*_m}^*(P)) &= (n_1, \dots, n_{m-1}, n_m -1).
\end{align*}
The value of the crystal operators $\ep_i$ can be easily determined by the Lusztig data. 

\begin{corollary}\label{corollary:crystaloperatorslusztig}
For a reduced word $\ui = (i_1, \dots, i_m)$ of $w_0$ and $P \in \P$, if $P$ has Lusztig data $n_\bullet^\ui$, then $\ep_{i_1}(P) = n_{1}^{\ui}$ and $\ep_{i^*_m}^*(P) = n_m^{\ui}$.
\end{corollary}

Theorem \ref{theorem:MVcrystalstructure} associates a unique MV polytope $\Pol(b)$ to each $b \in B(\infty)$, where $\Pol(b_0)$ is the polytope consisting of the point $\mu_e$. We can also use the Saito reflection on the bicrystal $B(\infty)$ to describe the polytope $\Pol(b)$. For the rest of the subsection, we follow {\cite[Section 3.3]{TheMVBasis}}. 
\begin{definition}\label{definition:Saitoreflection}
Define the map $\tilde{\sigma}_i :\{ b \in B(\infty) : \ep_i(b) =0\} \rightarrow \{b \in B(\infty): \ep_i^*(b)=0\}$ by $\tilde{\sigma}_i(b) = \left(\tf_i \right)^{\varphi_i^*(b)} \left( \te^*_i\right) ^{\ep_i^*(b)} (b)$. The \emph{Saito reflection} is the map 
\[
\sigma_{i}: B(\infty) \rightarrow \{ b \in B(\infty): \ep_{i}^*(b)=0\}
\] defined by $\sigma_{i}(b) =\tilde{\sigma}_i ((\te_i)^{\ep_i(b)}b) $.

Similarly, define $\tilde{\sigma}_i^*: \{ b \in B(\infty): \ep_i^*(b)=0 \} \rightarrow \{ b \in B(\infty): \ep_i(b)=0\}$ by  $\tilde{\sigma_i}^*(b) = \left(\tf_i^* \right)^{\varphi_i(b)} \left( \te_i\right) ^{\ep_i(b)} (b)$. Define the \emph{$*$-Saito reflection} as the map
\[
\sigma_i^*: B(\infty) \rightarrow \{ b \in B(\infty): \ep_{i}(b)=0\}
\]
defined by $\sigma_i^*(b) = \tilde{\sigma}^*_i \left( ({\te^*_{i}})^{\ep_i^*(b)} b\right)$. 
\end{definition}
Note that by definition, $\ep^*_i(\sigma_i(b))=0$ and $\ep_i(\sigma_i^*(b))=0$. Also, $\tilde{\sigma}_i^* = \tilde{\sigma}_i^{-1}$ by {\cite[Corollary 3.4.8]{SaitoReflection}}. The operators $\sigma_i, \sigma_i^*$ satisfy the same braid relations as the simple reflections $s_i$ thus for any $w\in W$, it is well defined to set $\sigma_w := \sigma_{i_1} \cdots \sigma_{i_m}$  where $\ui$ is a reduced word of $w$.

\begin{lemma}[{\cite[Proposition 3.4.7]{SaitoReflection}}, {\cite[Property (L3)]{PreprojectiveMVpolytopesBK}}]\label{lemma:saitoandlusztig}
Let $b \in B(\infty)$ and let $n_\bullet^\ui$ be the Lusztig data of $\Pol(b)$ with respect to $\ui = (i_1, \dots, i_m)$. Consider $\uj = (i_2, \dots, i_m,i_1^*)$. Then 
\[
n_\bullet^\uj(\Pol(\sigma_{i_1}(b))) = (n_2, n_3, \dots, n_m, 0)
\]
Consider $\underline{k} = (i_m^*, i_1, i_2, \dots, i_{m-1})$. Then
\[
n_\bullet^{\underline{k}}(\Pol( \sigma_{i^*_m}^*(b)) = (0, n_1, \dots, n_{m-1})
\]
\end{lemma}

Using this lemma, the Lusztig data can be computed by composing crystal operators with certain Saito reflections. 

\begin{corollary}\label{equation:crystalsepsilonaction}
For $b \in B(\infty)$ suppose $\Pol(b)$ has vertex data $(\mu_\bullet)$. Then for every $w\in W$ and $j \in I$ such that $ws_j > w$, 
\[
\mu_{ws_j} - \mu_w = \ep_j(\sigma_{w^{-1}}(b)) w \alpha_j^\vee
\]
\end{corollary}

\begin{proof}
Consider $\Pol(b)$ with vertex data $(\mu_\bullet)$. For $w \in W$ and $j\in I$ such that $ws_j > w$, there is a reduced word $\ui$ such that $s_{i_1} \cdots s_{i_{\ell(w)}} = w$ and $s_{i_{\ell{w}+1}} = s_j$. By definition, $\mu_{ws_j} - \mu_{w} = n_{\ell(w)+1}(\Pol(b)) \cdot w \alpha_j^\vee$, where $n_\bullet^\ui(P) = (n_1, \dots, n_m)$. 

Recall that $\sigma_{w^{-1}}= \sigma_{s_{i_{\ell(w)}}} \sigma_{s_{i_{\ell(w)-1}}} \cdots \sigma_{s_{i_2}} \sigma_{s_{i_1}}$. By Lemma \ref{lemma:saitoandlusztig}, then $\Pol(\sigma_{w^{-1}}(b))$ has Lusztig data $(n_{\ell(w)+1}, \dots, n_m, 0, \dots, 0)$ with respect to the reduced word $(j, i_{\ell(w)+2}, \dots, i_m, i_1^*, \dots, i_{\ell(w)}^*)$. By Corollary \ref{corollary:crystaloperatorslusztig}, $\ep_j (\sigma_{w^{-1}}(b)) = \ep_j(\Pol(\sigma_{w^{-1}}(b)) = n_{\ell(w)+1}(\Pol(b))$. 
\end{proof} 

This corollary allows us to write $\mu_w(b)$ in a closed form. Note that the map $\tilde{\sigma}_i$ has the property that $\wt(\tilde{\sigma}_i(b)) = s_i \wt(b)$ so it follows that $\wt (\sigma_i(b)) = s_i (\wt( {e_i^*}^{\ep_i^*(b)}(b))) = s_i (\wt(b) + \ep_i(b) \alpha_i^\vee)$. For non-trivial $w = s_{i_1} \cdots s_{i_m}$, by inductively applying this equality we have
\begin{align*}
\wt(\sigma_{w^{-1}}(b)) = \sum_{k=1}^m \ep_{i_k} (\sigma_{s_{i_m} \cdots s_{i_k}}(b)) s_{i_{k}} \cdots s_{i_m} \alpha_{i_k}^\vee.
\end{align*}
As $\mu_e=0$, it follows from Corollary \ref{equation:crystalsepsilonaction} that
\begin{align*}
\mu_{s_{i_1} \cdots s_{i_m}} (b) &= \sum_{k=1}^m \mu_{s_{i_1} \cdots s_{i_k}} - \mu_{s_{i_1} \cdots s_{i_{k-1}}} = \sum_{k=1}^m \ep_{i_k} (\sigma_{s_{i_m} \cdots s_{i_{k-1}}} (b)) s_{i_1} \cdots s_{i_{k-1}} \alpha_{i_k}^\vee  = w \cdot \wt(\sigma_{w^{-1}}(b)).
\end{align*}
Thus the vertex data $(\mu_\bullet(b))$ of $\Pol(b)$ can be explicitly determined by the Saito reflection where
\begin{equation}
\mu_w(b) = w \cdot \wt(\sigma_{w^{-1}}(b)) - \wt(b).
\end{equation}
Note that we shift by $\wt(b)$ so that $\mu_e(b) = \wt(b) - \wt(b)=0$.

\subsection{Preprojective algebra modules}\label{section:preprojective}

We give a very brief background on preprojective algebra modules and the associated MV polytope. This section is needed to prove the \emph{generalized diagonal relations} of Section \ref{section:generalizeddiagonals}. 

In this section, we restrict to the case that $G$ is a simply-laced algebraic group. First, we start with some general definitions. 

\begin{definition}\label{definition:quiver}
A \emph{quiver} $Q =(I, E, s, t)$ consists of a vertex set $I$, an arrow set $E$, a source map $s: E \rightarrow I$ and a target map $t: E \rightarrow I$. We write the arrow $\alpha \in E$ as $\alpha: i \rightarrow j$, where $i = s (\alpha)$ and $j = t(\alpha)$. 

Define $E^* = \{ \alpha^*: \alpha \in E\}$ where $s(\alpha^*) = t(\alpha)$ and $t(\alpha^*) = s(\alpha)$. Let $\overline{Q} = (I, E \sqcup E^*, s, t)$ be the double quiver. 
\end{definition}

The \emph{path algebra} of $\overline{Q}$ over $\C$ is the algebra $\C \overline{Q}$. Consider the ideal $J$ generated by $\sum_{\alpha \in E} (\alpha \alpha^* - \alpha^* \alpha)$. 

\begin{definition}\label{definition:pathalgebra}
The \emph{preprojective} of $\overline{Q}$ over $\C$, denoted by $\Lambda(Q)$, is the quotient of $\C \overline{Q}$ by the ideal $J$. 
\end{definition}

A $\Lambda(Q)$-module $M$ is an $I$-graded vector space $\bigoplus_{i \in I} M_i$ with maps $M_\alpha: M_{s(\alpha)} \rightarrow M_{t(\alpha)}$ for each $\alpha \in E \sqcup E^*$ which satisfy
\[
\sum_{\alpha \in E, \, t(\alpha) =i} M_{\alpha} M_{\alpha^*} - M_{\alpha^*} M_{\alpha} =0
\] 
for each $i \in I$. 

We consider a few special $\Lambda(Q)$-modules. For $i \in I$, let $S_i$ be the 1-dimensional module concentrated at the vertex $i$, where all arrows act as zero. Let $I_i$ be the annihilator of $S_i$. For any $w \in W$ we define $I_w := I_{i_1} \cdots I_{i_m}$, where $\ui$ is a reduced word of $w$. Note that this is independent of the choice of $\ui$ and thus is well-defined. 

For a module $M$, the $i$-socle is the largest submodule of $M$ which is isomorphic to $S_i^{\oplus k}$ for some $k\in \N$, while the $i$-head is the largest quotient of $M$ which is isomorphic to $S_i^{\oplus k}$ for some $k \in \N$. In fact,  $\mathrm{soc}_iM \cong \Hom_{\Lambda(Q)}(\Lambda(Q)/I_i, M)$ and $\mathrm{hd}_i M \cong (\Lambda(Q)/I_i) \otimes_{\Lambda(Q)} M$. 

Let $G$ be a simply-laced complex algebraic group. Fix $Q$ to be an orientation of the Dynkin diagram associated to the simple coroots of $G$ and set $\Lambda := \Lambda(Q)$. For $M$ a $\Lambda$-module, we can define \emph{dimension vector} as $\dimvec M = \sum_{i\in I} \dim M_i \alpha^\vee_i$, which is contained in the coroot lattice $Q^\vee$. By \cite{PreprojectiveMVpolytopesBK}, we can associated a GGMS polytope to a $\Lambda$-module $M$ by
\[
\Pol(M) := \conv \{ \dimvec M - \dimvec N: N \subset M \text{ is a submodule}\}.
\]
By {\cite[Theorem 5.4]{AffineMVpolytopes}}, for any $w \in W$ we define the submodules $M^w \subseteq M$ as the image of the map $I_w \otimes_{\Lambda}\Hom_{\Lambda}(I_w, M) \rightarrow M$. By Remark 5.19 (i), $\Pol(M)$ will have vertex data $(\mu_w)_{w \in W}$ where  $\mu_w = \dimvec M -  \dimvec M^{w}$. For certain modules $M$, $M^w$ and $M^{s_i w}$ are closely related. 

\begin{lemma}[{\cite[Lemma 2]{ReflexionsBGK}}]
For $w \in W$, consider $i \in I$ such that $s_iw>w$. For $M$ a finite-dimensional $\Lambda$-module, if $\Ext_\Lambda^1(S_i, M)=0$ then $M^{s_iw} \cong I_i \otimes_{\Lambda}M^w$. 
\end{lemma}

Finally, a result of Crawley-Boevey tells us that we can switch the rolls of $S_i$ and $M$ in the previous lemma. 

\begin{lemma}[{\cite[Lemma 1]{OntheExceptionalCB}}]
For any $\Lambda$-modules $X, Y$, $\dim \Ext_{\Lambda}^1(X,Y)=\dim \Ext_{\Lambda}^1 (Y,X)$.
\end{lemma}

Finally, we define the subset of $\Lambda$-modules $\mathscr{T}^w$. 

\begin{definition}
Let $\mathscr{T}^w$ to be the set of $\Lambda$-modules $M$ such that $M^w = M$. 
\end{definition}

By \cite[Remark 5.5 (ii)]{AffineMVpolytopes}, this is the same category $\mathscr{T}^w$ defined in \cite{AffineMVpolytopes} and is also the category of modules $\mathcal{C}_{w^{-1}w_0}$ defined in \cite[Definition 2.5]{Menard:Richardsonvarieties}. Note that for $M \in \mathscr{T}^w$, the vertex data of $\Pol(M)$ will satisfy $\mu_w = \mu_e$. 

\section{Combinatorial data of MV polytopes of highest vertex $w$}\label{section:combinatorialdata}

In this section, we define a subset of MV polytopes, called MV polytopes of highest vertex $w$, and show that these polytopes only have vertices labelled by elements bounded by $w$ in the Bruhat order. First, we introduce the definition of an MV polytope of highest vertex $w$. 

\begin{definition}
Fix $w \in W$. Let $P$ be an MV polytope with vertex data $(\mu_\bullet)$. We say $P$ is an MV polytope of highest vertex at most $w$ if $\mu_w = \mu_{w_0}$. Denote by $\Pw$ the set of MV polytopes of highest vertex $w$. 
\end{definition}

\begin{remark}
Recall in Section \ref{section:preprojective} we define the set of MV polytopes associated to $\mathscr{T}^w$ as the set of $\Lambda$-modules $M$ such that $M^w = M$. By \cite[Proposition 5.33]{Menard:Richardsonvarieties}, $\P_{ww_0}$ is the set of MV polytopes associated to the modules in $\mathscr{T}^w$ (under a reflection by $w_0$ and a shift to make $\mu_e =0$).
\end{remark}

\begin{example}\label{example:B2highestvertexs212}
Consider MV polytopes associated to the group of type $B_2$. For $w = s_2s_1s_2$, $\P_{s_2s_1s_2}$ is the set of polytopes such that $\mu_{s_2s_1s_2} = \mu_{w_0}$, see Figure \ref{figure:B2MVpolytopes} for an example. This condition will also imply that $\mu_{s_1s_2} = \mu_{s_1s_2s_1}$. In Section \ref{section:generalizeddiagonals} we will explore how the condition $\mu_w = \mu_{w_0}$ affects the vertex data of a rank 2 polytope. 

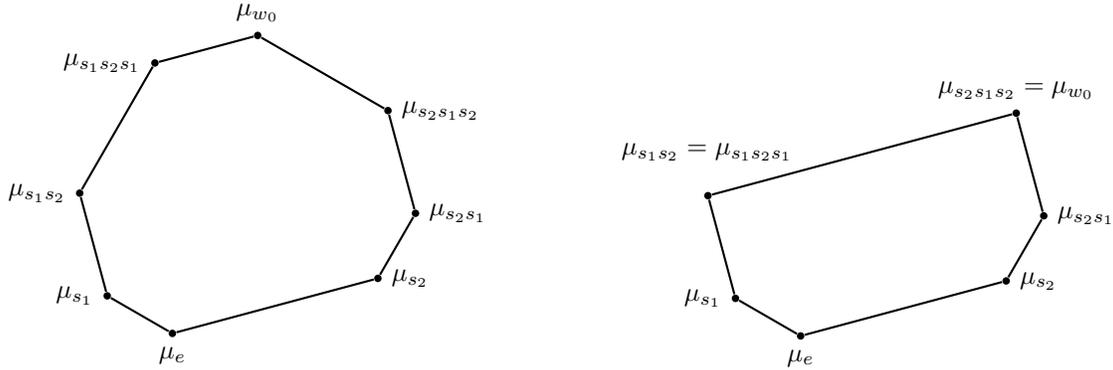
\begin{figure}[h]
\centering
\begin{subfigure}{.5\linewidth}
  \centering
  $\begin{tikzpicture}[rotate=60]
\node[My style] (SE) at (0,0) [label=below:$\mu_{e}$] {};
\node[My style] (S2) at (0,1)[label=left:$\mu_{s_1}$] {};
\node[My style] (S21) at (1,2)[label=left:$\mu_{s_1s_2}$] {};
\node[My style] (S1) at (2,-2)[label=right:$\mu_{s_2}$] {};
\node[My style] (S12) at (3,-2)[label=right:$\mu_{s_2s_1}$] {};
\node[My style] (S212) at (3,2)[label=left:$\mu_{s_1s_2s_1}$] {};
\node[My style] (S121) at (4,-1)[label=right:$\mu_{s_2s_1s_2}$] {};
\node[My style] (SW0) at (4,1)[label=above:$\mu_{w_0}$] {};

\draw[thick] (SE) -- (S1) -- (S12) -- (S121) -- (SW0) -- (S212) -- (S21) --  (S2) -- (SE);

\end{tikzpicture}$
\end{subfigure}%
\begin{subfigure}{.5\linewidth}
  \centering
$\begin{tikzpicture}[rotate=60]
\node[My style] (SE) at (0,0) [label=below:$\mu_{e}$] {};
\node[My style] (S2) at (0,1)[label=left:$\mu_{s_1}$] {};
\node[My style] (S21) at (1,2) {};
\node[My style] (S1) at (2,-2)[label=right:$\mu_{s_2}$] {};
\node[My style] (S12) at (3,-2)[label=right:$\mu_{s_2s_1}$] {};
\node[My style] (S121) at (4,-1)[label=above:{$\mu_{s_2s_1s_2}= \mu_{w_0}$}] {};

\node at (1.5, 2.3) {$\mu_{s_1s_2}=\mu_{s_1s_2s_1}$};
\node at (4,1)[label=above:{\color{white}$ \mu_{w_0}$}] {};

\draw[thick] (SE) -- (S1) -- (S12) -- (S121) --  (S21) --  (S2) -- (SE);

\end{tikzpicture}$
\end{subfigure}
\caption{A standard $B_2$ polytope (left) and a $B_2$ polytope of highest vertex $s_2s_1s_2$ (right)}
\label{figure:B2MVpolytopes}
\end{figure}
\end{example}

A reduced word $\ui$ for $w_0$ gives a  minimal path in the polytope of $P$ beginning at $\mu_e$ and ending at $\mu_{w_0}$. If this path passes through $\mu_w$, then the condition $\mu_w = \mu_{w_0}$ forces the Lusztig data $n_\bullet^\ui(P)$ to be zero in the coordinates after $\ell(w)$.  More precisely, we can show that every vertex which appears after $\mu_w$ in such a minimal path will necessarily be equal to $\mu_w$. 

\begin{lemma}\label{lemma:muw-muvinQ+}
Let $P$ be an MV polytope with vertex data $(\mu_w)_{w\in W}$. For $v, w \in W$, if $w \le_R v$ then $\mu_v - \mu_w \in Q^\vee_+$.
\end{lemma}

\begin{proof}
If $w \leq_R v$, then there exists a reduced word $(i_1, \dots, i_{\ell(v) - \ell(w)})$ such that $w s_{i_1} \dots s_{i_{\ell(v)-\ell(w)}} = v$. 
By (\ref{equation:lusztigdatafromvertexdata}) in  Section \ref{section:MVpolytopes}, 
\[
\mu_{v} - \mu_w = \sum_{k=1}^{\ell(v)-\ell(w)} c_k w s_{i_1} \dots s_{i_{k-1}} \alpha_{i_k}^\vee
\]
for coefficients $c_k \geq 0$. Since each $w s_{i_1} \dots s_{i_{k-1}}$ is reduced,  $w s_{i_1} \dots s_{i_{k-1}} \alpha_{i_k}^\vee$ is a positive coroot.  Thus $\mu_v - \mu_w$ is a non-negative sum of positive coroots so $\mu_v - \mu_w \in Q^\vee_+$.  
\end{proof}

\begin{lemma}\label{lemma:vw=w}
Fix $w \in W$ and suppose $P \in \Pw$. For $v \in W$, if $w \leq_R v$, then $\mu_{w} = \mu_v$. 
\end{lemma}

\begin{proof}
Suppose $v$ is such that $w \leq_R v$. By Lemma \ref{lemma:muw-muvinQ+}, $\mu_v - \mu_w \in Q^\vee_+$ and $\mu_{w_0} - \mu_v \in Q^\vee_+$. Thus 
\[
0 = \mu_{w_0} - \mu_w = (\mu_{w_0} - \mu_v) + (\mu_v - \mu_w)
\]
But the sum of non-zero points in $Q^\vee_+$ is still a non-zero point in $Q^\vee_+$ and hence the only possible values of $\mu_{w_0}- \mu_v$ and $\mu_v - \mu_w$ are zero. Thus $\mu_w = \mu_v = \mu_{w_0}$. 
\end{proof}

By the definition of the Lusztig data and its relation to the vertices (see (\ref{equation:lusztigdatafromvertexdata})), this lemma allows us to characterize $\Pw$ in terms of its Lusztig data with respect to certain reduced words. 
 
\begin{corollary}\label{corollary:vw=w}
Fix $w\in W$. The following conditions are equivalent:
 \begin{enumerate}[label=(\roman*)]
 \item $P \in \Pw$,
 \item There exists a reduced word $\ui$ of $w_0$ with $(i_1, \dots, i_{\ell(w)})$ a reduced word for $w$ such that the Lusztig data $n_\bullet^\ui(P)$ have $n_k =0$ for $k \geq \ell(w)$,
 \item For every reduced word $\ui$ of $w_0$ with $(i_1, \dots, i_{\ell(w)})$ a reduced word for $w$, the Lusztig data $n_\bullet^\ui(P)$ have $n_k =0$ for $k \geq \ell(w)$.
\end{enumerate}
\end{corollary}

Recall that an MV polytope $P$ is determined by its BZ data $(M_\gamma)_{\gamma \in \Gamma}$. We characterize the BZ data for $P \in \Pw$.

\begin{lemma}\label{lemma:BZdatahighestvertexw}
The collection $(M_\gamma)_{\gamma \in \Gamma}$ is the BZ datum of an MV polytope with highest vertex $w$ exactly when
\begin{enumerate}[label=(\roman*)]
 \item $(M_\gamma)_{\gamma \in \Gamma}$ is the BZ datum of an MV polytope,
 \item \label{condition:vanishing} There exists a reduced word $\uj = (j_{1}, \dots, j_{k})$ of $w^{-1}w_0$ such that for $\ell = 0, \dots, k-1$,
\begin{align}
  M_{w \cdot w_\ell^\uj \, s_{i_{\ell+1}}\omega_{i_{\ell+1}}} + M_{w \cdot w_\ell^\uj \, \omega_{i_{\ell+1}}} = - \sum_{j \neq i_{\ell+1}} a_{j,{i_{\ell+1}}} M_{w \cdot w_{\ell+1}^\uj \, \omega_j}. \label{equation:vanishing}
\end{align}
\end{enumerate}
\end{lemma}

\begin{proof}
Consider $P \in \P$ with vertex data $(\mu_\bullet)$ and BZ data $(M_\bullet)$. The only thing we need to show is that \ref{condition:vanishing} is equivalent to $\mu_w=\mu_{w_0}$. 

Suppose that $\mu_w = \mu_{w_0}$. For any reduced word $\uj$ of $w^{-1}w_0$, $w \leq_R w \cdot w_\ell^\uj$ for $0 \leq \ell \leq k$. By Lemma \ref{lemma:vw=w} it follows that $\mu_w = \mu_{w w_1^\uj} = \cdots = \mu_{w w_{k-1}^\uj} = \mu_{w_0}$. Thus $\mu_{w w_{\ell}^\uj} = \mu_{w w_{\ell+1}^\uj}$ for $0\leq \ell \leq k-1$, which is equivalent to (\ref{equation:vanishing}) by Lemma \ref{lemma:edgesandLusztig}. 

Suppose $\uj = (j_{1}, \dots, j_{k})$ is a reduced word for $w^{-1} w_0$ such that \ref{condition:vanishing} holds. As (\ref{equation:vanishing}) is equivalent to $\mu_{w w_{\ell}^\uj} = \mu_{w w_{\ell+1}^\uj}$ and this holds for $0 \leq \ell \leq k-1$, then $\mu_w = \mu_{ww^\uj_{1}} = \cdots = \mu_{ww^\uj_k} = \mu_{w_0}$ and so $P\in \Pw$. 
\end{proof}

Lemma \ref{lemma:BZdatahighestvertexw} and Corollary \ref{corollary:vw=w} both only give information about the structure of the polytope $P \in \Pw$ along the minimal paths from $\mu_e$ to $\mu_{w_0}$ that pass through the vertex $\mu_w$. To understand the whole structure of $P$, we need to understand the Lusztig data along any minimal path from $\mu_e$ to $\mu_{w_0}$. 

We will prove that for every $P \in \Pw$ with vertex data $(\mu_\bullet)$,  $\mu_v = \mu_{v_w}$ for some well defined element $v_w$. The proof is organized as follows. In Section \ref{section:bruhatintervals}, we define this element $v_w$ for $v, w \in W$. In Section \ref{section:generalizeddiagonals}, we outline the generalized diagonal relations and see how these relations completely determine the vertex data for rank 2 polytopes. In Section \ref{section:saitoreflectionresults}, we show that the Saito reflection acts on $\Pw$ in a useful way and finally, in Section \ref{section:lusztigdata} we will show exactly where the Lusztig data are zero for an arbitrary reduced word of $w_0$.

\subsection{Intersections of Bruhat intervals}\label{section:bruhatintervals}

In this section, we will investigate the intersections of intervals in the Bruhat order with intersections in the weak Bruhat order. 

First we recall some definitions and properties of Coxeter groups. Using the length, we can define the \emph{left descent set} and \emph{right descent set} respectively as
\begin{align*}
D_L(w) = \{ s_i \in S: \ell(s_iw) < \ell(w)\}, && D_R(w) = \{ s_i \in S: \ell(ws_i) < \ell(w)\}. 
\end{align*}
Note that the left and right descent sets can be defined via the weak orders: $s_i \in D_L(w) \iff s_i \leq_R w$ and $s_i \in D_R(w) \iff s_i\leq_L w$.  We can relate the weak Bruhat orders to the length function in the following way. 

\begin{proposition}[{\cite[Proposition 3.1.2]{CombinatoricsofCoxeter}}]\label{proposition:weakbruhats}
For $u, w \in W$, $u \leq_R w \iff \ell(u) + \ell(u^{-1}w) = \ell(w)$ and $u \leq_L w \iff \ell(wu^{-1}) + \ell(u) = \ell(w)$. 
\end{proposition}

For convenience, if $u\leq_R w$, we will say that $u$ is an \emph{initial word} of $w$, while we will say that $u$ is a \emph{terminal word} of $w$ if $u \leq_L w$. We will also say for $v, w \in W$, $v \cdot w$ is a \emph{reduced product} if $\ell(vw) = \ell(v) + \ell(w) $ (note that this is equivalent to $v \leq_L v \cdot w$ and $w \leq _R v \cdot w$). 

As $W$ is finite, there is a unique longest element $w_0$. This element has a special property for any reduced decomposition into two elements. 

\begin{lemma}\label{lemma:DL(w)cupDR=s}
For any $x,y \in W$ such that $w_0 = x \cdot y$ is a reduced product, then $D_R(x) \cap D_L(y) = \emptyset$ and $D_R(x) \cup D_L(w^{-1}w_0) = S$. 
\end{lemma}

\begin{proof}
By the conditions of $\ell(x) + \ell(y) = \ell(w_0)$ and $w_0 = x \cdot y$, then $D_R(x) \cap D_L(y) = \emptyset$. 

Suppose there exists $s \not\in D_R(x) \cup D_L(y)$. Then $x \cdot s \cdot y$ is an element of length $\ell(w_0) +1$, which contradicts the maximality of $w_0$. 
\end{proof}

Finally, Coxeter groups have three important properties that we will make use of multiple times throughout this section:
\begin{theorem}[{\cite[Proposition 2.2.7, Theorem 1.5.1, Theorem 3.3.1]{CombinatoricsofCoxeter}}]\label{theorem:weylproperties}
For $W$ a Coxeter group, 
\begin{enumerate}[align=left]
\item[\textbf{Lifting Property:}] Suppose $u < w$ and $s_i \in D_L(w)\backslash D_L(u)$. Then $u \leq s_iw$ and $s_iu \leq w$. 
\item[\textbf{Exchange Property:}] Let $w = s_{i_1} s_{i_2} \cdots s_{i_k}$ be a reduced expression. If $\ell(s_iw) \leq \ell(w)$ for $s_i \in S$, then $s_iw = s_{i_1} s_{i_2} \cdots \widehat{s_{i_j}} \cdots s_{i_k}$, where $\widehat{s_{i_j}}$ means that this term is deleted. 
\item[\textbf{Word Property:}] Every two reduced words for $w$ can be connected via a sequence of braid relations. 
\end{enumerate}
\end{theorem}

In Section \ref{section:lusztigdata}, we will prove that for $P \in \Pw$ with vertex data $(\mu_v)_{v \in W}$, $P = \conv \{ \mu_v: v \in W,  v \leq w\}$. The main result will be to explicitly describe the map $W \rightarrow [e,w]$ which arises by sending $v \mapsto u$ if for every $P \in \Pw$, $\mu_v = \mu_u$.  By Lemma \ref{lemma:vw=w}, we know that for $v \geq_R w$, $v \mapsto w$. When $v \not\geq_R w$, this map is slightly more complicated.

\begin{example}\label{example:vertexcollapsing}
Consider $G$ of type $A_3$. The simple coroots are $\alpha^\vee_1, \alpha^\vee_2, \alpha^\vee_3$ and the Weyl group is given by the presentation
$ W = \langle s_1, s_2, s_3 : (s_1s_3)^2=1, (s_1s_2)^3 = 1, (s_2s_3)^3 = 1, s_1^2 = s_2^2 = s_3^2 =1 \rangle$. Let $w = s_1s_2s_3$ and consider $P \in \Pw$ with Lusztig data $(1,1,1,0,0,0)$ associated to the reduced word $\ui = (1,2,3,1,2,1)$. This polytope has the following form:
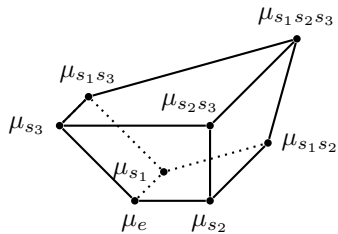
\[
\begin{tikzpicture}

\node[My style] (SE) at (0, 0, 0, 0) [label=below:$\mu_{e}$] {};
\node[My style] (S3) at (-1, 1,0, 0) [label=left:$\mu_{s_3}$] {};
\node[My style] (S2) at (1, 0, 0, -1)[label=below:$\mu_{s_2}$] {};
\node[My style] (S23) at (1, 1,0, -2)[label=above: $\mu_{s_2s_3}{\color{white}111}$]{};
\node[My style] (S1) at (0, 0,-1, 1)[label=left:$\mu_{s_1}$] {};
\node[My style] (S13) at (-1, 1,-1, 1)[label=above:$\mu_{s_1s_3}$] {};
\node[My style] (S12) at (1, 0,-2, 1)[label=right: $\mu_{s_1s_2}$]{};
\node[My style] (S123) at (1, 1,-3, 1) [label=above: $\mu_{s_1s_2s_3}$]{};

\draw[thick] (SE) -- (S3) ;
\draw[thick, dotted] (SE) --  (S1) -- (S12); 
\draw[thick, dotted] (S1) -- (S13) -- (S3);
\draw[thick] (SE) -- (S2) -- (S23) ;
\draw[thick] (S123) -- (S23) -- (S3) -- (S13) -- (S123) -- (S12) -- (S2);

\end{tikzpicture}
\]
Note that the vertices are indeed labelled by the set $\{ v \in W: v \leq w\}$. 
For $v\in W$ larger than $w$, the relations  on the vertices $\mu_v$ are:
\begin{align*}
\mu_{s_2s_3s_1s_2s_1} = \mu_{s_2s_3}, && \mu_{s_1s_3s_2s_1} = \mu_{s_1s_3}, && \mu_{s_3s_2s_1} = \mu_{s_3}, && \mu_{s_1s_2s_1} = \mu_{s_1s_2}, && \mu_{s_2s_1} = \mu_{s_2}.
\end{align*}
Notice that if $\mu_v = \mu_u$ then $u \leq_R v$. 
\end{example}

Suppose for $v \in W$, $u$ is the Weyl group element such that $u \leq w$ and $\mu_v = \mu_u$ for every $P \in \Pw$. By examining the previous example, we expect two conditions on $u$: first, we expect that $u\leq_R v$; equivalently, this says there must be a minimal path from $\mu_e$ to $\mu_{w_0}$ in the polytope that passes through both the vertices $\mu_u$ and $\mu_v$. Second, we expect that $u$ is the longest element such that $u \leq w$ and $u \leq_R v$. First we prove that this element is well-defined. To do this, we will need a result of Bj\"{o}rner and Wachs. 

For $x,y \in W$, we will say $z$ is a \emph{minimal upper bound} for $x$ and $y$ if $x,y \leq z$  and for any $z' \in W$ such that $x, y \leq z' \leq z$,  then $z = z'$.

\begin{theorem}[{\cite[Theorem 3.7, Theorem 4.4]{Coxeter1987}}]\label{theorem:theorem3.7coxeter}
Fix $v \in W$. Let $x,y \in [e,v]_R$ and suppose $z$ is a minimal upper bound of $x$ and $y$. Then $z \in [e,v]_R$. 
\end{theorem}

\begin{lemma}\label{lemma:longestexists}
For every $v, w \in W$, the set $[e,v]_R \cap [e,w]$ has a unique element of longest length. 
\end{lemma}

\begin{proof}
As $[e,v]_R \cap [e,w]$ is a finite set, there exists an element of longest length. Suppose there exists two distinct elements $x, y$ of longest length.  

Consider the set $[x,w] \cap [y,w]$. As this set is finite, there exists an element $z$ (not necessarily unique) of minimal length. This element $z$ has the property that for any $z'\in W$ such that $z'\leq z$, $x\leq z'$ and $y \leq z'$, then $\ell(z) =\ell(z')$ by the minimality of $z$ and hence $z = z'$. We apply \cite[Theorem 3.7]{Coxeter1987} (see Theorem \ref{theorem:theorem3.7coxeter}), so $z \leq_R v$ as well. Thus $z \in [e,v]_R \cap [e,w]$ but $\ell(z) > \ell(x) = \ell(y)$, which contradicts that $x, y$ are of longest length. 
\end{proof}

Since $x \leq_R v \iff x^{-1} \leq_L v^{-1}$ and $x \leq w \iff x^{-1} \leq w^{-1}$ then there is a bijection $[e,v^{-1}]_R \cap [e,w^{-1}] \rightarrow [e,v]_L \cap [e,w]$ by $x \mapsto x^{-1}$. As $\ell(x^{-1}) = \ell(x)$, this lemma also holds for the left Bruhat order. 

\begin{corollary}\label{corollary:longestexists}
For every $v, w \in W$, the set $[e,v]_L \cap [e,w]$ has a unique element of longest length. 
\end{corollary}

Lemma \ref{lemma:longestexists} ensures the following definition is well-defined. 

\begin{definition}
For any $v, w \in W$, denote $v_w$ to be the unique element of maximal length in $[e,v]_R \cap [e,w]$.
\end{definition}

This element is closely related to the Demazure product.  For $w\in W$ and $s_i \in S$, let $s_i* w := \max\{ w, s_i w\}$, where the maximum is the element in the set of maximal length. The \emph{Demazure product} can be defined recursively by $s_{i_1} * \dots * s_{i_k} := s_{i_1} * (s_{i_2} * \dots * s_{i_k})$. This product is associative and well-defined by \cite[Proposition 3.1]{Heckeproduct}.

\begin{proposition}[{\cite[Proposition 6.4]{Demazurepolytopes}}]
For $v, w \in W$, $v*w = \max\{ xw: x\leq v\}$.
\end{proposition}

Using the same proof technique as Proposition 6.4, we can relate the Demazure product to the weak orders.

\begin{proposition}\label{proposition:generalProp6.4}
For $v, w \in W$, 
\begin{align}
v*w &= \max\{ xw: x \leq v \text{ and } \ell(xw) = \ell(x) + \ell(w)\} \label{equation:demazurexw} \\
&= \max \{ vy: y\leq w \text{ and } \ell(vy) = \ell(v) + \ell(y)\}. \label{equation:demazurevx} 
\end{align}
Moreover, $v*w =xw = vy$ where $x$ is the maximal length element such that $x\leq v$ and $xw$ is reduced and $y$ is the maximal length element such that $y \leq w$ and $vy$ is reduced. 
\end{proposition}

\begin{proof}
If $v=e$, then $e*w = w$ and clearly (\ref{equation:demazurexw}) holds. We proceed by induction. 

For $v \neq e$, there exists $s_i \in D_L(v)$. As $\ell(s_i v) = \ell(v) - 1$, then by induction $(s_i v) * w = xw$ for some $x \leq s_i v$ and $\ell(xw) = \ell(x) + \ell(w)$. Since $s_i* (s_iv) = v$, then 
\[
v * w = s_i * (s_i v) * w = s_i  * (xw)
\]
because $*$ is associative. If $\ell(xw)> \ell(s_ixw)$, then $v*w = xw$ for $x \leq s_i v \leq v$ and $\ell(xw) = \ell(x) + \ell(w)$. Otherwise, $\ell(s_ixw) = \ell(xw) + 1$ so $v*w = s_ixw$. Note that $\ell(s_ixw) = 1 + \ell(xw) = \ell(x) + \ell(w) +1$ so $s_i \not\in D_L(x)$ and $\ell(s_ix) = \ell(x) +1$. This implies $\ell(s_ixw) = \ell(s_ix) + \ell(w)$. Finally, by the Lifting Property, $s_ix \leq v$ and (\ref{equation:demazurexw}) holds. A similar proof works for (\ref{equation:demazurevx}). 

The maximal length element in the set $\{ xw : x \leq v \text{ and } \ell(xw) = \ell(x) + \ell(w)\}$ must occur when $\ell(x)$ is of maximal length. Thus $v*w = xw$ where $x$ is the maximal length element such that $x \leq v$  and $xw$ is reduced. By an identical argument, $v*w = vy$ for $y$ the maximal length element such that $y \leq w$ and $vy$ is reduced. 
\end{proof}
It immediately follows that $w \leq_L v*w$ and $v \leq_R v*w$ by  Proposition \ref{proposition:weakbruhats}. 

To relate $v_w$ to the Demazure product, we first need the following lemma. 

\begin{lemma}\label{lemma:reducedproductandbruhat}
For $u,v \in W$, the following conditions are equivalent:
\begin{enumerate}[label=\arabic*)]
\item $u \cdot v$ is a reduced product
\item $v \leq_R u^{-1}w_0$
\item $u \leq_L w_0v^{-1}$
\end{enumerate}
\end{lemma}

\begin{proof}
By definition, $u \cdot v$ is reduced if and only if $\ell(uv) = \ell(u) + \ell(v)$. Then
\begin{align*}
\ell(v) + \ell(v^{-1} u^{-1}w_0) & = \ell(v) + \ell(w_0) - \ell(v^{-1}u^{-1}) = \ell(v) + \ell(w_0) -\ell(uv) \\ & = \ell(w_0) - \ell(u)  = \ell(u^{-1}w_0)
\end{align*}
so by Proposition \ref{proposition:weakbruhats}, $v \leq_R u^{-1}w_0$. A similar proof works for $u \leq_L w_0v^{-1}$. 
\end{proof}

This lemma is applying the fact that for the longest element, $w_0 = w \cdot (w^{-1}w_0)$ is a reduced product for any $w$, so $w^{-1}w_0$ multiplied on the left with any terminal word of $w$  will also be reduced. 

\begin{proposition}\label{proposition:demazureproduct}
Fix $w \in W$. For any $v \in W$, $v_w = (vw_0)((w_0v^{-1})* w)$, where $*$ is the Demazure product.
\end{proposition}

\begin{proof}
By Proposition \ref{proposition:generalProp6.4}, $(w_0v^{-1}) * w = (w_0v^{-1}) \cdot x$ where $x$ is the maximal length element such that $x \leq w$ and $(w_0v^{-1})\cdot x$ is reduced.  By Lemma \ref{lemma:reducedproductandbruhat}, $w_0 v^{-1} \cdot x$ is reduced if and only if $x\leq_R v$. Thus $x$ is the maximal length element such that $x \leq w$ and $x\leq_R v$ so by definition, $x = v_w$. 
\end{proof}

We could alternatively take $(vw_0)((w_0v^{-1})*w)$ as the definition of $v_w$ and the uniqueness of $v_w$ will be automatic as this product is well-defined. For our purposes, it is useful to use the Bruhat orders to define $v_w$ but this connection to the Demazure product simplifies some of the proofs. 

As $v_w$ is an initial word of $v$, then $v= v_w \cdot (v_w^{-1} v)$ is a reduced product. The next lemma shows how the terminal word $v_w^{-1}v$ of $v$ relates to $w^{-1} w_0$.

\begin{lemma}\label{lemma:vwinverse}
Fix $w \in W$. For every $v \in W$, $v_w^{-1} v \leq_R w^{-1} w_0$.  
\end{lemma}

\begin{proof}
Note that $v_w^{-1} v \leq_R w^{-1}w_0 \iff v_w^{-1} v w_0 \geq_R w^{-1} \iff  (vw_0)^{-1} v_w \geq_L w$. But $v_w = (vw_0)((vw_0)^{-1} * w)$ so that $(vw_0)^{-1} v_w =  (vw_0)^{-1} * w$. A consequence of Proposition \ref{proposition:generalProp6.4} is that $(vw_0)^{-1}v_w =(vw_0)^{-1} * w \geq_L w$ as desired. 
\end{proof}

\subsection{Generalized diagonals}\label{section:generalizeddiagonals}

In this section we prove two technical lemmas, which state the \emph{generalized diagonal relations} on MV polytopes. These relations are inspired by the diagonal relations in the rank 2 case, see the discussion at the end of \cite[Section 3]{MVpolytopes} for more details. These inequalities are interesting because they relate vertices of the form $\mu_{w}$ and $\mu_{s_jw}$ which are vertices that do not necessarily share a face of the polytope (see Figure \ref{figure:A3generalizeddiagonals}). On the other had, the tropical Pl\"{u}cker relations only give relations amongst vertices with a shared face.  

\begin{figure}[h]
  \centering
$
\begin{tikzpicture}
\node[My style] (SE) at (0, 0, 0, 0) [label=below:$\mu_{e}$] {};
\node[My style] (S3) at (-1, 1, 0, 0 ) [label=below:$\mu_{s_3}$] {};
\node[My style] (S32) at (-1,2, 0, -1)  {};
\node[My style] (S2) at (1.5, 0, 0, -1.5) [label=below:$\mu_{s_2}$] {};
\node[My style] (S23) at (1.5, 1, 0, -2.5) {};
\node[My style] (S232) at (0.5, 2, 0, -2.5)[label=below:$\mu_{s_2s_3s_2}$]{};
\node[My style] (S1) at (0,0, -1, 1) [label=above:$\mu_{s_1}$] {};
\node[My style] (E6) at (-1, 1, -1, 1) {};
\node[My style] (S321) at (-1, 3,-1, -1) {};
\node[My style] (S21) at (2.5, 0, -1, -1.5) {};
\node[My style] (S231) at (2.5, 1, -1, -2.5) {};
\node[My style] (E9) at (0.5, 3,-1, -2.5) {};
\node[My style] (S12) at (0.5, 0,-1.5, 1) {};
\node[My style] (E11) at (-1, 1.5,-1.5, 1) {};
\node[My style] (E12) at (-1, 3,-1.5,- 0.5) {};
\node[My style] (B1) at (2.5, 0,-1.5, -1) {};
\node[My style] (B2) at (2.5, 1.5,-1.5, -2.5) {};
\node[My style] (B3) at (1, 3,-1.5, -2.5) {};
\node[My style] (B4) at (0.5, 1,-2.5, 1) {};
\node[My style] (B5) at (0, 1.5,-2.5, 1) [label=above:$\mu_{s_1s_2s_3s_2}$] {};
\node[My style] (B6) at (0, 3,-2.5, -0.5) {};
\node[My style] (B7) at (2.5, 1,-2.5, -1) {};
\node[My style] (B8) at (2.5, 1.5,-2.5, -1.5) {};
\node[My style] (SW0) at (1, 3,-2.5, -1.5)[label=above:$\mu_{w_0}$]{};

\draw[thick] (E12) -- (S321) -- (S32);
\draw[thick] (S23) -- (S232) -- (S32)-- (S3) -- (SE) -- (S2) -- (S21) -- (B1) -- (B7) -- (B8) -- (B2) -- (S231);
\draw[thick] (S2) -- (S23) -- (S231) -- (S21);
\draw[thick] (S232) -- (E9) -- (S321);
\draw[thick] (E12) -- (B6) -- (SW0) -- (B3) -- (E9);
\draw[thick] (B3) -- (B2);
\draw[thick] (SW0) -- (B8);
\draw[thick, dotted] (SE) -- (S1) -- (S12) -- (B1);
\draw[thick, dotted] (E6) -- (S1);
\draw[thick, dotted] (S3) -- (E6) -- (E11) -- (E12);
\draw[thick, dotted] (E11) -- (B5) -- (B4) -- (S12);
\draw[thick, dotted] (B6) -- (B5);
\draw[thick, dotted] (B4) -- (B8);

\filldraw[draw=black, fill=black, opacity=0.2]  (S232.center) -- (2, 2, -1.5, -2.5) -- (2, 2, -2.5, -1.5) -- (0, 2, -2.5, 0.5) -- (-1, 2, -1.5, 0.5) -- (-1, 2, 0, -1) -- cycle;

\end{tikzpicture}
$
\caption{A generalized diagonal with strict inequality}
\label{figure:A3generalizeddiagonals}
\end{figure}
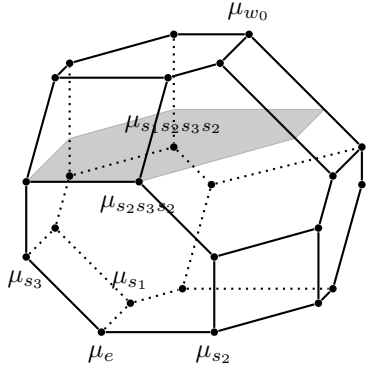

The first lemma requires the use of preprojective algebra modules, see Section \ref{section:preprojective} for more details. We recall a few definitions. For $G$ a simply-laced complex algebraic group, let $\Lambda$ be the preprojective algebra associated to the double quiver of an orientation of the Dynkin diagram of the coroots of $\g$. For a $\Lambda$-module $M$, define the submodules $M^w$ as the image of the map $I_w \otimes_{\Lambda} \Hom_{\Lambda}(I_w, M) \rightarrow M$. The associated MV polytope is given by
\[
\Pol (M) = \conv \{ \mu_w: w \in W\}, \text{ where } \mu_w = \dimvec M - \dimvec M^w. 
\]

The following proof of the simply-laced case is due to Pierre Baumann. We thank him for allowing its inclusion in this text. 

\begin{lemma}\label{lemma:generaldiagonalrelations-simplylaced}
Assume $G$ is simply-laced. Let $P$ be an MV polytope with vertex data $(\mu_w)_{w \in W}$. For every $w \in W$, $s_j \in D_L(w)$,  the inequality $\langle \mu_{w} - \mu_{s_j w}, \omega_k\rangle \leq 0$ holds for every $ k \neq j$. 
\end{lemma}

\begin{proof}
Let $M$ be the $\Lambda$-module associated to the polytope $P$, i.e. $\dimvec M - \dimvec M^w = \mu_w$ for $w \in W$. We want to prove that $\mu_w - \mu_{s_j w} =  \dimvec M^{s_jw}  - \dimvec M^w = n\alpha_j^\vee - \beta$ for $n \in \N$ and $\beta \in Q^\vee_+$. 

First, suppose that $\Ext^1(M, S_j)=0$. By \cite[Lemma 2]{ReflexionsBGK}, $M^{w} = I_j \otimes_{\Lambda}M^{s_jw}$. Consider the short exact sequence
\[
0 \rightarrow I_j \rightarrow \Lambda \rightarrow \Lambda/I_j \rightarrow 0.
\] 
As the tensor product is right exact, by applying the functor $\otimes_{\Lambda} M^{s_jw}$ we get the long exact sequence $ \dots \rightarrow I_j \otimes_{\Lambda}M^{s_jw} \rightarrow \Lambda \otimes_{\Lambda} M^{s_jw} \rightarrow \left( \Lambda/I_j \right) \otimes_{\Lambda} M^{s_jw}$. Note that $\left( \Lambda/I_j \right) \otimes_{\Lambda} M^{s_jw} = M^{s_iw}/I_j$; by definition, $M^{s_jw}/I_j = \text{hd}_j (M^{s_jw}) = S_j^{\oplus n}$ for some $n \in \N$. Thus we have the resulting exact sequence
\[
0 \rightarrow \ker(\phi) \rightarrow M^{w} \xrightarrow{\phi} M^{s_jw} \rightarrow S_j^{\oplus n} \rightarrow 0, 
\]
with the dimension vectors $\dimvec M^{s_jw} + \dimvec \ker(\phi) = \dimvec M^{w} + n \alpha_i$.
As $\dimvec \ker(\phi) \in Q^\vee_+$, the claim holds for $M$ of this form. 

Suppose $M$ is a general $\Lambda$-module. Let $N$ be the maximal extension of $M$ by $S_j$, i.e. take $m \in \N$ such that
\[
0 \rightarrow S_j^{\oplus m} \rightarrow N \rightarrow M \rightarrow 0.
\]

By the proof of \cite[Lemma 2]{ReflexionsBGK}, $0 \to S_j^{\oplus m} \to N^{s_jw} \to M^{s_jw} \to 0$ is exact and thus $\dimvec{N^{s_jw}} = \dimvec{M^{s_jw}} + m \alpha^\vee_j$. Also, as the composition $N^{w} \rightarrow N \rightarrow M \rightarrow M/M^{w}$ is zero, then $\dimvec N/N^{w} \geq \dimvec M/M^{w}$. Thus there exists $\gamma \in Q^\vee_+$ such that  $\dimvec N/N^w -\dimvec M/M^w = \gamma$. Then
\begin{align*}
\dimvec M^{w} - \dimvec N^{w}  = \gamma + \dimvec M - \dimvec N = \gamma - m \alpha_j^\vee
\end{align*}
Finally, the difference between the dimension vector of $M^w$ and $M^{s_jw}$ is as follows:
\begin{align*}
\dimvec M^{s_jw} - \dimvec M^w &= \dimvec N^{s_jw} - \dimvec N^w - (\dimvec M^w - \dimvec N^w) - (\dimvec N^{s_jw} - \dimvec M^{s_jw})\\
& = n\alpha^\vee_j - \beta - (\gamma - m \alpha^\vee_j) - (m\alpha^\vee_j) = n \alpha^\vee_j - \beta - \gamma
\end{align*}
Since $\mu_w = \dimvec M - \dimvec M^w$, then $\mu_{w} - \mu_{s_jw} = n\alpha^\vee_j - \beta - \gamma$ so $\langle\mu_w - \mu_{s_jw}, \omega_k \rangle  =- \langle \beta + \gamma, \omega_k \rangle \leq 0$ for every $k \neq j$. 
\end{proof}

Now, we implement the folding technique of to prove the general case. We will follow the notation used in \cite{JiangSheng2017}. 
    
 Let $G$ be a simply-laced algebraic group. Consider a bijection $\sigma: I \rightarrow I $  with $a_{ij} = a_{\sigma(i)\sigma(j)}$.  This induces a \emph{Dynkin diagram automorphism} on $G$ by $\sigma: G \rightarrow G$ such that $\sigma(x_{\pm i}(a)) = x_{\pm\sigma(i)} (a)$.  Let $G^\sigma$ be the fixed point group on $G$ and call the pair $(G, G^\sigma)$ a \emph{symmetric pair}. 
 
  A Dynkin diagram automorphism $\sigma$ induces an action on the weight and coweight lattices by $\sigma(\alpha_i) = \alpha_{\sigma(i)}$ and $\sigma(\alpha_i^\vee) = \alpha_{\sigma(i)}^\vee$. It  also induces a group automorphism on $W$ by $\sigma(s_i) = s_{\sigma(i)}$. We set $ W^\sigma := \{ w \in W: \sigma(w) = w\}$. 
 
 Denote by $\overline{\g}$ to be the Lie algebra of $G^\sigma$. Let $\overline{W}$ be the Weyl Group of $\overline{\g}$, generated by simple reflections $\hs_i$. There is a group isomorphism $\Theta: \overline{W} \rightarrow W^\sigma$ defined by
 \[
 \Theta(\hs_i) = s_i^\sigma:= \prod_{t=0}^{k_i-1} s_{\sigma^t(i)}
 \]
 where $k_i$ is the number of elements in the $\sigma$-orbit of $i$.  

Now, we will consider the $\sigma$-invariant MV polytopes of $G$.  Denote $\P$ to be the set of MV polytopes for $\g$. The diagram automorphism $\sigma$ induces an action on $\P$ by
 \[
 \sigma(P) := \text{conv} \{ \sigma^{-1}(\mu_{\sigma(w)}): w\in W\}.
 \]
 If $\sigma(P) = P$, we call \emph{$P$ $\sigma$-invariant}. Denote the set of $\sigma$-invariant MV polytopes by $\P^\sigma$ and let  $\overline{\P}$ be the set of MV polytopes for $\overline{\g}$. There is an identification between these two sets of polytopes. 
 
\begin{theorem}[{\cite[Theorem 3.10]{Hong2007MVSymmetricPair}}, {\cite[Theorem 6.2]{JiangSheng2017}}]\label{theorem:dynkindiagramautomorphism}
For $P \in \P^\sigma$ with vertex data $(\mu_{w})_{w\in W}$, define $\Phi(P) = \text{conv} \{ \overline{\mu}_{\overline{w}}: \overline{w}\in \overline{W}\}$, where $\overline{\mu}_{\overline{w}} := \mu_{\Theta({\overline{w}})}$. The map $\Phi: \P^\sigma \rightarrow \overline{\P}$ is a bijection. 
\end{theorem}

Now we have the machinery to prove the non-simply-laced case. 

\begin{lemma}\label{lemma:generaldiagonalrelations-nonsimplylaced}
Assume $G$ is non-simply-laced. Let $P$ be an MV polytope with vertex data $(\mu_w)_{w \in W}$. For every $w \in W$, $s_j \in D_L(w)$, the inequality $\langle \mu_{w} - \mu_{s_j w}, \omega_k\rangle \leq 0$ holds for every $ k \neq j$. 
\end{lemma}

\begin{proof}
Let $\sigma$ be a Dynkin diagram automorphism of $G$ and let $\P^\sigma$ be the set of $\sigma$-invariant MV polytopes of $G$. Let $\overline{\P}$ be the set of MV polytopes associated to $G^\sigma$, the fixed point group of $\sigma$. Recall by Theorem \ref{theorem:dynkindiagramautomorphism}, $\Phi: \P^\sigma \rightarrow \overline{\P}$ is a bijection, where $P$ with vertex data $(\mu_w)_{w\in W}$ is sent to $\overline{P}$ with vertex data $(\mu_{\Theta(\overline{w})}) = (\overline{\mu}_{\overline{w}})$. 

Let $\overline{P} \in \overline{\P}$. Consider $\overline{w} \in \overline{W}$ arbitrary. Let $\hs_j \in D_L(\overline{w})$ and $k \in \overline{I}$ such that $k \neq j$. We want to show that $\langle \overline{\mu}_{\overline{w}} - \overline{\mu}_{\hs_j \overline{w}}, \overline{\omega}_k \rangle \leq 0$.

Since $\Phi$ is a bijection, there exists a $P \in \P^\sigma$ such that 
$\overline{\mu}_{\overline{w}} = \mu_{\Theta(\overline{w})}$. Then for the vertices of $\overline{P}$, 
\[
\overline{\mu}_{\overline{w}} - \overline{\mu}_{\hs_j \overline{w}} = \mu_{\Theta(\overline{w})} - \mu_{\Theta(\hs_j\overline{w})} =  \mu_{\Theta(\overline{w})} - \mu_{s_j^\sigma \Theta(\overline{w})} .
\]

Note that $s_i^\sigma$ depends on the number of elements in the $\sigma$-orbit of $i$, which can only equal $1,2$ or $3$. We consider the case where there are $3$ elements in the orbit. Then
\begin{align*}
\mu_{\Theta(\overline{w})} - \mu_{s_j^\sigma {\Theta(\overline{w})}} &= \mu_{\Theta(\overline{w})} -\mu_{s_{\sigma^2(j)} {\Theta(\overline{w})}}+\mu_{s_{\sigma^2(j)} {\Theta(\overline{w})}}-\mu_{s_{\sigma(j)} s_{\sigma^2(j)} {\Theta(\overline{w})}} \\
& {\color{white}=} + \mu_{s_{\sigma(j)} s_{\sigma^2(j)} {\Theta(\overline{w})}} - \mu_{s_{j} s_{\sigma(j)} s_{\sigma^2(j)} {\Theta(\overline{w})}}\\
\implies \langle \mu_{\Theta(\overline{w})} - \mu_{s_j^\sigma {\Theta(\overline{w})}}, \omega_k \rangle & = \langle \mu_{\Theta(\overline{w})} -\mu_{s_{\sigma^2(j)} {\Theta(\overline{w})}} , \omega_k \rangle + \langle \mu_{s_{\sigma^2(j)} {\Theta(\overline{w})}}-\mu_{s_{\sigma(j)} s_{\sigma^2(j)} {\Theta(\overline{w})}} , \omega_k \rangle  \\
& {\color{white}=} + \langle \mu_{s_{\sigma(j)} s_{\sigma^2(j)} {\Theta(\overline{w})}} - \mu_{s_{j} s_{\sigma(j)} s_{\sigma^2(j)} {\Theta(\overline{w})}} , \omega_k \rangle 
\end{align*}
Since $k \in \overline{I}$ but $\{ j, \sigma(j), \sigma^2(j)\}\cap \overline{I} = \{j\}$, then $k \neq \sigma(j)$ or $\sigma^2(j)$. Thus we can apply the simply-laced case to each term on the right side of the above equation, and hence $\langle \mu_{\Theta(\overline{w})} - \mu_{s_j^\sigma {\Theta(\overline{w})}},\omega_k \rangle \leq 0$. As $\overline{\omega}_k$ is the restriction of $\omega_k$ to the subspace $\h^\sigma$, then $\langle \overline{\mu}_{\overline{w}} - \overline{\mu}_{\hs_j \overline{w}}, \overline{\omega}_k \rangle=\langle \mu_{\Theta(\overline{w})} - \mu_{s_i^\sigma {\Theta(\overline{w})}}, \omega_k \rangle  \leq 0$. For the cases with $1$ or $2$ elements in the $\sigma$-orbit, we will simply have fewer terms on the right side of the above equation. 
\end{proof}

Recall we define $*:I \rightarrow I$ where $i^*$ is the index such that $s_i w_0 = w_0 s_{i^*}$. For $w = s_{i_1} \cdots s_{i_m}$, we define $w^* = s_{i_1^*} \cdots s_{i^*_m}$. 

\begin{lemma}\label{lemma:s_jw}
Fix $w \in W$. For every $P \in \Pw$, and for every $s_j \in D_L(w)$, $\mu_{s_j w} = \mu_{w_0 s_{j^*}}$. 
\end{lemma}

\begin{proof}
First, as $w \leq_R w_0$ then $s_j w \leq_R s_j w_0 = w_0 s_{j^*}$ so by Lemma \ref{lemma:muw-muvinQ+}, $\mu_{w_0 s_{j^*}} - \mu_{s_jw} \in Q^\vee_+$. Thus for a reduced word $\ui$ of $w^{-1}w_0$ ending in $j^*$, 
\begin{align}\label{equation:muw0sj*-musjw}
\mu_{w_0s_{j^*}} - \mu_{s_jw} = \sum_{r=1}^{\ell(w_0) - \ell(w)-1} c_r (s_jw) s_{i_1} \dots s_{i_{r-1}} \alpha_{i_r}^\vee
\end{align}
for $c_r \geq 0$ and positive coroots $(s_j w) s_{i_1}\cdots s_{i_{r-1}}\alpha_{i_r}^\vee$. It follows that $\langle \mu_{w_0 s_{j^*}} - \mu_{s_jw}, \omega_k \rangle \geq 0$ for all $k \in I$. 

By the generalized diagonal relations, $\langle \mu_w - \mu_{s_jw}, \omega_k \rangle \leq 0$ for $k \neq j$ and hence $\langle \mu_{w_0} - \mu_{s_j w}, \omega_k \rangle \leq 0$ as well. As $ (w_0 s_{j^*}) \cdot \alpha_{j^*}^\vee = \alpha_j^\vee$, then $\mu_{w_0} - \mu_{w_0 s_{j^*}}  \in \Z\alpha^\vee_j$, so for all $k \neq j$, 
\begin{align*}
0=\langle \mu_{w_0} - \mu_{w_0 s_{j^*}}, \omega_k \rangle = \langle \mu_{w_0} - \mu_{s_j w}, \omega_k \rangle  - \langle \mu_{w_0 s_{j^*}} - \mu_{s_j w}, \omega_k \rangle.
\end{align*}
Both of these terms must be zero for all $k \neq j$, so $\mu_{w_0 s_{j^*}} - \mu_{s_j w} \in \Z \alpha_j^\vee$. As each $(s_jw) s_{i_1} \cdots s_{i_{r-1}}\alpha_{i_r}^\vee$ is a distinct positive coroot for every $r$ and 
\[
(s_j w) s_{i_1} \cdots s_{i_{\ell(w_0) - \ell(w)}} \alpha_{i_{\ell(w_0) - \ell(w)}}^\vee = (s_j w) w^{-1}w_0 \alpha_{j^*}^\vee = (w_0 s_{j^*}) \alpha_{j^*}^\vee = \alpha_j^\vee
\]
then it follows that $(s_jw)s_{i_1} \cdots s_{i_{r-1}}\alpha_{i_r}^\vee \neq \alpha_j^\vee$ for every $r< \ell(w_0) - \ell(w)$. If any $c_r \neq 0$ in (\ref{equation:muw0sj*-musjw}), then $\alpha_j^\vee$ is a positive sum of positive coroots, which contradicts that $\alpha_j^\vee$ is a simple coroot. Thus $c_r=0$ for all $1 \leq r \leq \ell(w_0) - \ell(w)-1$ and $\mu_{s_0s_{j^*}} = \mu_{s_jw}$.
\end{proof}

When $G$ is of rank 2, then Lemma \ref{lemma:s_jw} completely determines the vertex data of a polytope in $\Pw$. To see this, consider $w = s_{i_1} s_{i_2} \dots s_{i_m} \in W$. As $G$ has two simple roots, there are only two simple reflections and so $w$ is an alternating product of $s_1$ and $s_2$. 

The existence of only two simple roots means that MV polytopes are 2-dimensional polygons. The two simple reflections generate two distinct reduced words for $w_0$: the alternating product $s_1s_2s_1\dots$ of length $\ell(w_0)$ and the alternating product $s_2s_1s_2\dots$ of length $\ell(w_0)$. These two reduced words give two minimals paths from $\mu_e$ to $\mu_{w_0}$ and correspond to the two sides of the polygon. 

For $P \in \Pw$, $\mu_w$ is on one side of the polygon and the vertex data for any vertex along this minimal path is described by Lemma \ref{lemma:vw=w}, i.e. if $v\geq_R w$, then $\mu_v = \mu_w$, otherwise $v \leq_R w$ and $\mu_v$ can be distinct. For $v \not\leq_R w$ and $v\not\geq_R w$, then $\mu_v$ is necessarily on the minimal path from $\mu_e$ to $\mu_{w_0}$ which does not contain $\mu_w$, and hence either $v \leq_R s_{i_1} w$ or $v \geq_R s_{i_1} w$. By Lemma \ref{lemma:s_jw}, $D_L(w) = s_{i_1}$ and so $\mu_{s_{i_1}w} = \mu_{w_0s_{i_1^*}}$. 

If $v \not\leq w$, then $v\not\leq_R w$ and $v \not \leq_R s_{i_1}w$, so it must follow that either $v \geq_R w$ or $v \geq_R s_{i_1}w$. For the first case, we have already shown $\mu_v = \mu_w$. For the second case, as $\mu_{v}$ is between the vertices $\mu_{s_{i_1}w}$ and $\mu_{w_0s_{i_1^*}}$, the equality $\mu_{s_{i_1}w} = \mu_{w_0 s_{i_1^*}}$ forces $\mu_{v} = \mu_{s_{i_1}w}$. It follows that on the side of the polygon which does not contain $\mu_w$, the highest vertex is labelled by $\mu_{s_{i_1}w}$ and the only possible distinct vertices are labelled by $v \leq w$. Hence $P = \conv \{ \mu_v : v \in W, v\leq w \}$. 

\subsection{Crystal action on $\Pw$}\label{section:saitoreflectionresults}

In this section, we will show that the Saito reflection behaves well with $\Pw$. First, we briefly recall the crystal structure of MV polytopes and the Saito reflection, see Section \ref{section:crystals} for more details. 

The set of MV polytopes has crystal structure $B(\infty)$. The MV polytope associated to $b \in B(\infty)$, denoted $\Pol(b)$, is given by the vertex data $(\mu_w(b))_{w\in W}$, where
\begin{align}\label{equation:saitoreflectionpolytope}
\mu_w(b)  = w \cdot \text{wt} (\sigma_{w^{-1}} (b)) - \wt(b).
\end{align}

The action of the Saito reflection on the crystal $B(\infty)$ has a known effect on the Lusztig data of $\Pol(b)$ for $b \in B(\infty)$. If $\Pol(b)$ has Lusztig data $(n_1, \dots n_m)$ associated to the reduced word $(i_1, \dots, i_m)$, then $\Pol(\sigma_{i_1}(b))$ has Lusztig data $(n_2, \dots, n_m, 0)$ associated to the reduced word $(i_2, \dots, i_m, i_1^*)$ while $\Pol(\sigma^*_{i_m}(b))$ has Lusztig data $(0, n_1, \dots, n_{m-1})$ associated to the reduced word $(i_m^*, i_1, \cdots, i_{m-1})$. The crystal operators are also related to the Lusztig data $n_\bullet^\ui(P)$ by $\ep_{i_m}(b)=n_m$ and $\ep_{i^*_1}^*(b) = n_1$.  

 \begin{lemma} \label{lemma:TFAEpol(b)}
 Fix $w \in W$. Let $\ui = (i_1, \dots, i_m)$ be a reduced word for $w^{-1}w_0$. For $b \in B(\infty)$, the following are equivalent:
 \begin{enumerate}[label=(\roman*)]
  \item $\Pol(b) \in \Pw$, \label{condition:TFAEpol(b)1}
  \item $\ep^*_{i^*_m}(b)=0$ and  $\ep^*_{i^*_k}(\sigma^*_{s_{i^*_{k+1}} \dots s_{i^*_{m}}}(b)) =0$ for $1 \leq k < m$, \label{condition:TFAEpol(b)2}
  \item $\sigma_{w^{-1}}(b) = b_0$. \label{condition:TFAEpol(b)3}
 \end{enumerate}
\end{lemma}

\begin{proof}
Extend $\ui$ to a reduced word $\ui' = (j_1, \dots, j_{\ell(w_0)-m} , i_1, \dots, i_m)$ of $w_0$. Denote the Lusztig data of $\Pol(b)$ associated to $\ui'$ by $(n_1, \dots, n_{\ell(w_0)-m}, N_1, \dots, N_m)$. 
By Corollary \ref{corollary:crystaloperatorslusztig}, we know that $\ep^*_{i^*_m}(b) = N_m$. For $1 \leq k <m$, consider the polytope $\Pol(\sigma^*_{s_{i^*_{k+1}} \dots s_{i^*_m}}(b))$. By Lemma \ref{lemma:saitoandlusztig}, this polytope has Lusztig data $(0, \dots, 0, n_1, \dots, n_{\ell(w_0)-m}, N_1, \dots, N_k)$ associated to the word $(i_{k+1}^*, \dots, i_m^*, j_1, \dots, j_{\ell(w_0) - m}, i_1, \dots, i_k)$. Then $\ep^*_{i^*_k}(\sigma^*_{s_{i_{k+1}} \dots s_{i_m}}(b)) = N_k$. 

Thus $\ep^*_{i^*_k}(\sigma_{s_{i^*_{k+1}} \cdots s_{i^*_m}}(b)) =0$ for $1 \leq k\leq m$ if and only if $N_k=0$ for $1 \leq k \leq m$. By Corollary \ref{corollary:vw=w}, this is equivalent to $\Pol(b) \in \Pw$ so \ref{condition:TFAEpol(b)1} is equivalent to \ref{condition:TFAEpol(b)2}. 

As $\Pol(b)$ has vertex data $\mu_v = v \cdot \wt(\sigma_{v^{-1}}(b)) - \wt(b)$ then
\[
\mu_w = \mu_{w_0} \iff w \cdot \wt (\sigma_{w^{-1}}(b)) = w_0 \cdot \wt(\sigma_{w_0^{-1}}(b)) \iff w\cdot \wt (\sigma_{w^{-1}}(b)) = 0
\] 
since $\sigma_{w_0}(b') = b_0$ for every $b' \in B(\infty)$. As $b_0$ is the unique element of weight zero, then the weight $\wt(\sigma_{w^{-1}}(b)) = 0 \iff \sigma_{w^{-1}}(b) = b_0$. Thus \ref{condition:TFAEpol(b)1} is equivalent to \ref{condition:TFAEpol(b)3}. 
\end{proof}

In the next two lemmas, we will show how the Saito reflection $\sigma^*_{j^*}$ acts on $P \in \Pw$. 

\begin{corollary}\label{corollary:saito}
Fix $w \in W$. For $b \in B(\infty)$, if $\Pol(b) \in \Pw$ and $s_j \in D_R(w^{-1} w_0)$, then $\Pol(\sigma^*_{j^*}(b)) \in \P_{s_{j^*} w}$. 
\end{corollary}

\begin{proof}
First, notice that the condition $s_j \in D_R(w^{-1}w_0)$ ensures there exists a reduced word  of $w_0$ 
$\ui = (i_1, \dots, i_{\ell(w)}, k_1, \dots, k_{m - \ell(w)-1}, j)$ such that $(i_1, \dots, i_{\ell(w)})$ is a reduced word of $w$ and $(k_1, \dots, k_{m - \ell(w)-1}, j)$ is a reduced word of $w^{-1}w_0$. 

Suppose that $\Pol(b)\in \Pw$. By Corollary \ref{corollary:vw=w}, the Lusztig data of $\Pol(b)$ is $(n_1, \dots, n_{\ell(w)}, 0, \dots, 0)$ with respect to $\ui$. By Lemma \ref{lemma:saitoandlusztig}, $\Pol(\sigma_{j^*}^*(b))$ has Lusztig data $(0, n_1, \dots, n_{\ell(w)}, 0, \dots, 0)$ with respect to the reduced word $(j^*, i_1, \dots, i_{\ell(w)}, k_1, \dots, k_{m - \ell(w)-1})$. Hence  $\Pol(\sigma_{j^*}^*(b)) \in \P_{s_j^*w}$ by Corollary \ref{corollary:vw=w}. 
\end{proof}

\begin{lemma}\label{lemma:saitoPw}
Fix $w \in W$. For $b \in B(\infty)$, if $\Pol(b) \in \Pw$ and $s_j \notin D_R(w^{-1} w_0)$, then $\Pol(\sigma^*_{j^*}(b)) \in \Pw$.
\end{lemma}

\begin{proof}
As $w_0 = w \cdot w^{-1} w_0$ is reduced, we also have the reduced product $w_0 = (w^{-1}w_0) \cdot w^*$. By setting $x = w^{-1}w_0$ and $y = w^*$, we can apply Lemma \ref{lemma:DL(w)cupDR=s} so that 
$D_R(w^{-1} w_0) \cap D_L(w^*) = \emptyset$ and $D_R(w^{-1} w_0) \cup D_L(w^*) = S$. Thus $s_j \not\in D_R(w^{-1}w_0) \iff s_{j^*} \in D_L(w)$. 

Consider $b\in B(\infty)$ such that $\Pol(b) \in \Pw$. Let $s_{j^*} \in D_L(w)$. By Lemma \ref{lemma:TFAEpol(b)}, to show $\Pol(\sigma_{j^*}^*(b)) \in \Pw$, it is enough to show that $\sigma_{w^{-1}}(\sigma_{j^*}^*(b)) = b_0$.

Let $\ui =(i_1, \dots, i_{m-1}, j)$ be a reduced word of $w_0$ such that $w = s_{j^*}s_{i_1} \cdots s_{i_k}$ for $k = \ell(w) -1$. Let $(n_1, \dots, n_{m-1}, n_m)$ be the Lusztig data of $\Pol(b)$ with respect to $\ui$. The polytope $\Pol\left(\sigma_{j^*}^*(b)\right)$ has Lusztig data $(0, n_1, \cdots, n_{m-1})$ for reduced word $(j^*, i_1, \dots, i_{m-1})$ so that $\Pol\left(\sigma_{j^*}\sigma_{j^*}^*(b)\right)$ has Lusztig data $(n_1, \dots, n_{m-1}, 0)$ with respect to $\ui$. Notice that 
\[
\sigma_{w^{-1}}(\sigma_{j^*}^*(b)) = \sigma_{(s_{j^*}w)^{-1}} (\sigma_{j^*}\sigma_j^*(b)) = \sigma_{s_{i_k} \cdots s_{i_2}s_{i_1}}(\sigma_{j^*}\sigma_j^*(b))
\] 
so that $\Pol\left(\sigma_{w^{-1}}(\sigma_{j^*}^*(b))\right)$ has Lusztig data $(n_{k+1}, \dots, n_{m-1}, 0, \dots, 0)$ with respect to the reduced word $(i_{k+1}, \dots, i_{m-1}, j, i_1^*, \dots, i_{k}^*)$. 

Since $\Pol(b) \in \Pw$, then  $\mu_{s_{i_1}\cdots s_{i_k}} = \mu_{s_{j^*} w} = \mu_{w_0s_j}$ by by the generalized diagonal relations of Lemma \ref{lemma:s_jw}. The relation between Lusztig data and vertices (see (\ref{equation:lusztigdatafromvertexdata})) implies that $n_{k+1} = \cdots n_{m-1} =0$ in this Lusztig data. Thus $\Pol\left( \sigma_{w^{-1}}(\sigma_{j^*}^*(b))\right) = \Pol(b_0) $. By the uniqueness of $\Pol(b)$, $\sigma_{w^{-1}}(\sigma_{j^*}^*(b))=b_0$ as desired.
\end{proof}

\subsection{Lusztig and vertex data of $\Pw$}\label{section:lusztigdata}

The goal of this section is to show that for any $P \in \Pw$,  $\mu_v = \mu_{v_w}$ for every $v \in W$. First, we need to investigate where the zeros in the Lusztig data are located. 

Let $\ui=(i_1, \dots, i_m)$ be a tuple. Consider two subwords of $\ui$, $\underline{a} = (i_{a_1}, \dots, i_{a_k})$ and $\underline{b} = (i_{b_1} \dots, i_{b_k})$. We say the subword $\underline{a}$ comes after $\underline{b}$ in the reverse-lexicographical order if for some $n$, $a_n<b_n$ and $a_j = b_j$ for every $j\geq n$. 

\begin{definition}
Let $\ui$ be a reduced word of $w_0$. For $w \in W$, define the \textbf{rightmost} subword $\ui^w$ as the first subword in the reverse-lexicographical ordering that is a reduced word of $w$. 
\end{definition}

The next two lemmas will show that this rightmost word for $w^{-1}w_0$ will always start with a reduced word for $v_w^{-1} v$. 

\begin{lemma}\label{lemma:intersection}
Fix $w \in W$. Let $\ui = (i_1, \dots, i_m)$ be a reduced word for $w_0$ and let $i^{w} = (i_{j_1}, \dots, i_{j_{\ell(w)}})$. 

For any terminal subword $\ui' = (i_k, i_{k+1}, \dots, i_{m})$ of $\ui$, the subword of $\ui$ indexed by the intersection $\{k, k+1, \dots, m\} \cap \{j_1, \dots, j_{\ell(w)}\}$ is a reduced word for the maximal length element in $[e, w]_L \cap [e,  s_{i_k} \cdots s_{i_{m}}]$.
\end{lemma}

\begin{proof}
We proceed by induction on $m+1-k$, the length of the terminal subword $\ui'$.  

Suppose $k =m$. If $[e,w]_L \cap [e, s_{i_m}] \neq \{e\}$, then the maximal element in this set is $s_{i_m}$. Then $s_{i_m} \in D_R(w)$ and hence $j_{\ell(w)} = m$ by definition. Thus the intersection $\{m\} \cap \{j_1, \dots, j_{\ell(w)}\} = m$ and the reduced word $(i_m)$ is a reduced word for $s_{i_m}$. If $[e,w]_L \cap [e, s_{i_m}]  = \{e\}$, then $s_{i_m} \not\in D_R(w)$ and so $\{m\} \cap \{j_1, \dots, j_{\ell(w)}\} = \emptyset$ which is a reduced word for the maximal element. 

Assume the hypothesis holds for the subword $(i_{k+1}, \dots, i_m)$ and let $y'$ be Weyl element given by the subword of $\ui$ indexed by $\{k+1, \dots, m\} \cap \{j_1, \cdots, j_{\ell(w)}\}$. By assumption, this is also the maximal element in $[e,w]_L \cap [e, s_{i_{k+1}} \cdots s_{i_m}]$. 

For $\ui' = (i_k, \dots, i_m)$, let $y$ be the Weyl element given by the subword of $\ui$ indexed by $\{k, k+1, \dots, m \} \cap \{j_1, \dots, j_{\ell(w)}\}$. Then either $y = y'$ or $y = s_{i_k} \cdot y'$ is a reduced product and so $\ell(y') \leq \ell(y)$. By definition of $\ui^w$, $y \leq_L w$ and hence $y \in[e,w]_L \cap [e,s_{i_k} \dots s_{i_m}]$. 

For any $x \in [e,w]_L \cap [e,s_{i_k} \dots s_{i_m}]$, either $x \in [e,w]_L \cap [e,s_{i_{k+1}} \dots s_{i_m}]$ or $x = s_{i_k} \cdot x'$ is a reduced product for some $x ' \in [e,w]_L \cap [e,s_{i_{k+1}} \dots s_{i_m}]$. In the first case, $\ell(x) \leq \ell(y') \leq \ell(y)$. In the second case, if $x' \neq y'$, then  $\ell(x) = \ell(x')+1 < \ell(y')+1$ so that $\ell(x) \leq \ell(y') \leq \ell(y)$. If $x'=y'$, then necessarily $x = s_{i_k} y'=y$ by above and $\ell(x) = \ell(y)$. Thus the length of every element in this intersection is bounded above by $\ell(y)$ and hence $y$ must be the unique of the maximal length element.
\end{proof}

Note that by the definition of $\ui^w$, the phrase ``the subword of $\ui$ indexed by'' in the previous lemma can be replaced with ``the terminal subword of $\ui^w$ indexed by''. 

\begin{lemma}\label{lemma:reducedwords}
Let $v,w \in W$. For every reduced word $\ui = (i_1, \dots, i_m)$ of $w_0$ such that $v_w = s_{i_1} \dots s_{i_{\ell(v_w)}}$ and $v = s_{i_1} \dots s_{i_{\ell(v)}}$, $\ui^{w^{-1}w_0}=(i_{\ell(v_w) + 1},\dots, i_{\ell(v)}, i_{j_{\ell(v)+1}} \dots, i_{j_{m + \ell(v_w) - \ell(w)}})$ for some indices $\ell(v)+1 \leq j_{\ell(v) + 1} \leq \dots \leq j_{m + \ell(v_w) - \ell(w)}\leq m$. 
\end{lemma}

\begin{proof}
Let $\ui$ be a reduced word of $w_0$ as in the statement of the lemma. To show that the word $\ui^{w^{-1}w_0}$ begins with a reduced word for $v_w^{-1}v$, we will show that $(v_w^{-1}v)^{-1} w^{-1}w_0$ is the longest length element of $[e,w^{-1}w_0]_L \cap [e,v^{-1}w_0]$. By Lemma \ref{lemma:intersection}, this says that the length $\ell(w^{-1} w_0) - \ell(v_w^{-1}v)$ terminal word of $\ui^{w^{-1}w_0}$ is a reduced word of $(v_w^{-1}v)^{-1}w^{-1}w_0$. But as $w^{-1}w_0 = \left(v_w^{-1}v \right) \cdot \left((v_w^{-1}v)^{-1}w^{-1}w_0\right)$ is a reduced product, this will imply the initial word of length $\ell(v_w^{-1}v)$ of $\ui^{w^{-1}w_0}$ must be a reduced word for $v_w^{-1}v$. 

\begin{claim} $(v_w^{-1}v)^{-1}w^{-1}w_0$ is the longest length element of $[e,w^{-1}w_0]_L \cap [e, v^{-1}w_0]$. 
\end{claim}

\begin{claimproof}
Note that $x \in [e, w^{-1}w_0]_L \cap [e,v^{-1}w_0] \iff x^{-1} \in [e, w_0 w]_R \cap [e, w_0 v]$. The longest element in this intersection is the Demazure product $(w_0 w w_0) ((w_0 w^{-1} w_0) * (w_0v))$ by Proposition \ref{proposition:demazureproduct}. We want to show that this product is equal to $w_0 w(v_w^{-1}v)$. 

By Proposition \ref{proposition:generalProp6.4},  $((w_0 w^{-1} w_0) * (w_0v)) = w_0 w^{-1} w_0 x$ where $x$ is the maximal length element such that $x \leq w_0 v$ and $(w_0 w^{-1} w_0) \cdot x$ is reduced. Recall that $a \mapsto w_0 a w_0$ is an automorphism of the weak and strong Bruhat orders. Then  $x\leq w_0 v \iff w_0 x w_0 \leq vw_0$. Also, by Lemma \ref{lemma:reducedproductandbruhat},
\[
(w_0w^{-1}w_0) \cdot x \text{ is reduced  } \iff x \leq_R w_0 w \iff w_0 xw_0 \leq_R ww_0 \iff w^{-1} \cdot (w_0xw_0) \text{ is reduced}
\]
Since $\ell(x) = \ell(w_0 xw_0)$, then $w^{-1} * (vw_0) = w^{-1} (w_0 xw_0)$ by the maximality of $x$. Thus
\begin{align*}
x = w_0 w (w^{-1} * (vw_0)) w_0 = w_0 w ((vw_0)^{-1} * w)^{-1} w_0 =  w_0 w (v_w^{-1} (vw_0)) w_0 = (w_0 w) (v_w^{-1} v)
\end{align*} 
Hence $w_0 w (v_w^{-1}v)$ is the longest length element in $[e, w_0 w]_R \cap [e,w_0 v]$ and so $(v_w^{-1}v)^{-1} w^{-1} w_0$ is the longest length element in $[e, w^{-1}w_0]_L \cap [e, v^{-1}w_0]$. 
\end{claimproof}

Now, by applying Lemma \ref{lemma:intersection}, the terminal subword of $\ui^{w^{-1}w_0}$ indexed by the intersection of the indices of $\ui^{w^{-1}w_0}$ with $\{\ell(v)+1, \dots, m\}$ is a reduced word of $(v_w^{-1}v)^{-1} w^{-1}w_0$. Thus this intersection is of length $\ell(w_0) + \ell(v_w) - \ell(w) - \ell(v)$ and is equal to $\{ j_{\ell(v)+1}, j_{\ell(v)+2}, \dots, j_{m + \ell(v_w) - \ell(w)}\}$ for some indices $\ell(v)+1 \leq j_{\ell(v)+1} \leq \cdots \leq j_{m + \ell(v_w)-\ell(w)} \leq m$.  As $(i_{\ell(v_w) + 1}, \cdots, i_{\ell(v)})$ is a reduced word of $v_w^{-1} v$, then the word $(i_{\ell(v_w) + 1}, \cdots, i_{\ell(v)}, i_{j_{\ell(v)+1}}, \dots, i_{j_{m - \ell(w)}}) $ is a reduced word for $w^{-1}w_0$ and must be the rightmost such word. 
\end{proof}

We will show that for any $P \in \Pw$, the Lusztig data of $P$ with respect to the reduced word $\ui$ will have zeros in the position of the subword $\ui^{w^{-1}w_0}$. 

\begin{proposition}\label{proposition:zeros}
Let $\ui = (i_1, \dots, i_m)$ by any reduced word of $w_0$. For any $w \in W$ and any $P \in \Pw$, the Lusztig data of $P$ with respect to $\ui$ will have zeros in the position of the subword $\ui^{w^{-1}w_0}$. 
\end{proposition}

\begin{proof}
Fix a reduced word $\ui = (i_1, \dots, i_m)$ of $w_0$. We proceed by induction on $\ell(w^{-1} w_0)$. 

When $\ell(w^{-1}w_0) = 1$, then $\ui^{w^{-1}w_0} = i_j$ for some $j$. If $j=m$, then $(i_1, \dots, i_{m-1})$ is a reduced word for $w$ and by Lemma \ref{lemma:vw=w}, $n^\ui_{m}=0$. If $j \neq m$, then $\sigma_{s_{i^*_{j+1}} \cdots s_{i^*_m}}^*(b) \in \Pw$ by Lemma \ref{lemma:saitoPw} so the reduced word $\ui' = (i_{j+1}^*, \dots, i_m^*, i_1, \dots, i_j)$ has ${\ui'}^{w^{-1}w_0}$ in the last position, hence $n_{i_j} =0$ by above.

Assume for $\ell(w^{-1}w_0) =k$, the zeros of the Lusztig data $n_\bullet^\ui$ are in the position $\ui^{w^{-1}w_0}$. Suppose $w$ is such that $\ell(w^{-1}w_0) = k+1$ and $n_\bullet^\ui$ is the Lusztig data with respect to $\ui$. If $i_j$ is the final coordinate of $\ui^{w^{-1}w_0}$, then $s_{i_{j+1}}, \dots, s_{i_m} \not\in D_R(w^{-1}w_0)$ so we can apply Lemma \ref{lemma:saitoPw} $j-1$ times so that the Lusztig data with respect to $(i_{j+1}^*, \dots, i_m^*, i_1, \dots, i_j)$ of the resulting polytope is $(0, \dots, 0, n_1, \dots, n_j)$. By the base case, $n_j=0$ and hence the Lusztig datum in the position of the final term of $\ui^{w^{-1}w_0}$ is zero. Now, apply $\sigma^*_{i_j^*}$ so that the Lusztig data with respect to $(i_j^*, i_{j+1}^*, \dots, i_m^*, i_1, \dots, i_{j-1})$ of the resulting polytope is $(0, 0, \dots, 0, n_1, \dots, n_{j-1})$. By Corollary \ref{corollary:saito}, this polytope is in $\P_{s_{i_j^*}w}$, where $\ell(s_{i_j^*}w) = \ell(w) -1$. Thus, by the induction assumption, the Lusztig data corresponding to the rest of the coordinates of $\ui^{w^{-1}w_0}$ will be zero. 
\end{proof}

\begin{example}\label{example:vertexcollapsing2}
Continuing Example \ref{example:vertexcollapsing}, we have $w = s_1s_2s_3$ and $w^{-1} w_0 = s_1s_2s_1 = s_2s_1s_2$.  

For a word $\ui$, the zeros in the Lusztig data are given by the rightmost appearance of a word of $w^{-1}w_0$. The location of these zeros in various reduced words of $w_0$ prove the vertex equalities in Example \ref{example:vertexcollapsing}.  

\renewcommand{\arraystretch}{1.5}

\begin{table}[h]
\centering
\begin{tabular}{l|l|l}
Reduced word & Lusztig data & Equality of vertices\\ \hline
$(1,2,3,1,2,1)$ & $(n_1, n_2, n_3, 0, 0, 0)$ & $\mu_{w} =\mu_{ws_1} = \mu_{ws_1s_2} = \mu_{w_0}$\\
$(2,3,1,2,1,3)$ & $(n_1, n_2, 0,0,0,n_6)$ & $\mu_{s_2s_3} =\mu_{s_2s_3s_1} = \mu_{s_2s_3s_1s_2} = \mu_{s_2s_3s_1s_2s_1}$\\
$(1,3,2,1,3,2)$ & $(n_1, n_2, 0, 0, n_5, 0)$ & $\mu_{s_1s_3} = \mu_{s_1s_3s_2} = \mu_{s_1s_3s_2s_1}$, $\mu_{ws_2s_1} = \mu_{w_0}$ \\
$(3, 2, 1, 3, 2, 3)$ & $(n_1, 0, 0, n_4, 0, n_6)$ & $\mu_{s_3} = \mu_{s_3s_1} = \mu_{s_3s_2s_1}$, $\mu_{s_3s_2s_1s_3} = \mu_{s_3s_2s_1s_3s_2}$ \\
$(1,2,1,3,2,1)$ & $(n_1, n_2, 0, n_4, 0, 0)$ & $\mu_{s_1s_2} = \mu_{s_1s_2s_1}$, $\mu_{ws_1} = \mu_{w s_1 s_2} = \mu_{w_0}$\\
$(2,1,3,2,1,3)$ & $(n_1, 0, n_3, 0, 0, n_6)$ & $\mu_{s_2} = \mu_{s_2s_1}$, $\mu_{s_2s_1s_3} = \mu_{s_2s_1s_3s_2} = \mu_{s_2s_1s_3s_2s_1}$\\
\end{tabular}
\label{table:zeros}
\caption{The zeros in the Lusztig data for $A_3$ MV polytopes}
\label{table:lusztigdata}
\end{table}
\renewcommand{\arraystretch}{1}
\end{example}

Finally, we can prove that the Lusztig data will have zeros in the positions between $\mu_{v_w}$ and $\mu_v$ for every Weyl group element $v$. 

\begin{theorem}\label{theorem:mu_w}
Fix $w \in W$. For every $P \in \Pw$ with vertex data $(\mu_v)_{v \in W}$, $\mu_v = \mu_{v_w}$ for every $v \in W$. 
\end{theorem}

\begin{proof}
Consider $P \in \Pw$. For $v \in W$, take a reduced word $\ui$ of $w_0$ such that both $v_w$ and $v$ are initial words, i.e. $v_w = s_{i_1} \dots s_{i_{\ell(v_w)}}$ and $v = s_{i_1} \dots s_{i_{\ell(v)}}$.  Then by Lemma \ref{lemma:reducedwords} and Lemma \ref{proposition:zeros}, we know that the Lusztig data associated to $\ui$ will have zeros in the subword $\ui^{w^{-1}w_0} = (i_{\ell(v_w) + 1}, \dots, i_{\ell(v)}, i_{j_{\ell(v)+1}}, \dots, i_{j_{m - \ell(v)}})$. Hence $\mu_{v_w} = \mu_v$. 
\end{proof}

\begin{corollary}\label{corollary:vertexdata}
Fix $w \in W$. For every $P \in \Pw$ with vertex data $(\mu_v)_{v \in W}$, $P = \conv \{ \mu_v: v \leq w\}$.
\end{corollary}

\begin{remark} 
The description of $\Pw$ given by Corollary \ref{corollary:vertexdata} suggests a relationship between $\Pw$ and extremal MV polytopes defined by \cite{DemazureCrystalNaito-Sagaki}. Naito and Sagaki prove that extremal MV polytopes can be explicitly described as $P_{w \cdot  \lambda} = \conv \{ v \cdot \lambda: v \leq w\}$, where $\lambda$ is a dominant coweight. Using the Lusztig data description of these extremal MV polytopes in that paper, we can see (up to a reflection by $w_0$ and a shift to make $\mu_e =0$), these polytopes are in $\Pw$. 

In \cite{Demazurepolytopes}, Besson, Jeralds and Keirs study the weight polytopes of Demazure modules and prove they are extremal MV polytopes. These polytopes can be described in the following ways:
\[
P_\lambda^w = \conv\{ g(v) \lambda : v \in W\} =\conv \{ v \lambda: v \leq W\} = \bigcap_{v\in W} C_v^{g(v) \lambda}
\]
where $g(v) = v(v^{-1} * w)$. By Proposition \ref{proposition:demazureproduct} proved above, $g(vw_0) = v_w$. Thus, under the identification $\mu_v = g(vw_0)\lambda - g(w_0)\lambda$, $P_\lambda^w \in \Pw$. 

Another related concept are the polytopes defined in \cite{Bruhatintervalpolytopes}. For $w \in S_n$, the \emph{Bruhat interval polytope} $Q_{e,w}$ is the polytope with vertex data $(\mu_v)_{v \leq w}$ given by $\mu_v  = (v(1), \dots, v(n))$. By extending the vertex data to $(\mu_v)_{v\in W}$ by $\mu_v = \mu_{v_w}$, we see this is a polytope of highest vertex $w$. 
\end{remark} 

\subsection{The dual fan}\label{section:dualfan}

A GGMS polytope can be characterized by its dual fan in relation to a standard fan, called the Weyl fan. To describe this relationship, first we define fans and dual fans of polytopes. 

Let $V$ be a real vector space and let $V^*$ be the dual space. A \emph{polyhedral cone} in $V$ is an finite intersection of closed linear half spaces. A \emph{fan} $\F$ of $V^*$ is a collection of polyhedral cones with the following properties:
\begin{enumerate}[label=(\roman*)]
 \item Every nonempty face of a cone in $\F$ is also a cone in $\F$,
 \item The intersection of any two cones in $\F$ is a face of both,
 \item The union $\bigcup \F = V^*$.
\end{enumerate}
A fan $\F_1$ is a \emph{coarsening} of $\F_2$ if every cone of $\F_1$ is a union of cones in $\F_2$. 

Define the Weyl fan $\mathcal{W}$ in $\t^*_\R$ as the fan generated by the maximal cones
\[
C_w^* = \{ \alpha \in \t^*_\R : \langle w \cdot \alpha_i^\vee, \alpha \rangle \geq 0, \forall i \in I\}.
\]

For any convex polytope $P\subset V$, we can define the support function of $P$ as $\psi_P: V^* \rightarrow \R$ by $\psi_P(\alpha) = \min_{x\in P} \langle x, \alpha \rangle$. Define the dual fan $\mathcal{N}(P) = \{ C_F^* : F \text{ is a face of }P\}$ in $V^*$, where 
\[
C_F^* = \{ \alpha \in V^*:  \langle v, \alpha \rangle = \psi_P(\alpha), \forall v \in F\}
\]
\begin{corollary}[{\cite[Corollary A.4]{MVpolytopes}}]\label{corollary:A4ofMV}
A GGMS polytope $P$ is a polytope in $Q^\vee$ whose dual fan $\mathcal{N}(P)$ is a coarsening of the Weyl fan $\mathcal{W}$. 
\end{corollary}

The dual fan is a useful tool to study the vertices and hyperplanes of $P$. Maximal cones of the dual fan correspond to vertices of the polytope. If $\mathcal{N}(P)$ is a coarsening of $\mathcal{W}$, then there is an surjection from $W$ to the set of vertices of $P$; in fact, this surjection determines the choice of labelling on the vertices $\mu_w$. 

Additionally, the defining rays of the maximal cones of the dual fan correspond to the codimension 1 faces of $P$. This correspondence defines a surjection from the chamber weights $\Gamma$ to the defining rays of the maximal cones of $\mathcal{N}(P)$. 

Using Theorem \ref{theorem:mu_w}, we would like to study the dual fan of polytopes in $\Pw$. Recall from Section \ref{section:dualfan} that the Weyl fan $\mathcal{W}$ is the fan of $\t_\R^*$ with  maximal cones $C_v^*$ for $v \in W$ defined by
\[
C^*_v = \{ \beta \in \t_\R^*: \langle v \cdot \alpha_i^\vee, \beta \rangle \geq 0,\, \forall i\}.
\]
Any GGMS polytope $P$ with vertex data $(\mu_\bullet)$ is given by $P=\bigcap_{v \in W} C_v^{\mu_v}$ where
\[
C_v^{\mu_v} = \lbrace x \in \t_\R : \langle x, v \cdot \omega_i \rangle \geq \langle \mu_v, v \cdot \omega_i \rangle, \forall i\rbrace.
\]
The dual fan of a GGMS polytope $P$ is $\mathcal{N}(P) = \{C_{F,P}^*: F \text{ is a face of }P\}$ such that 
 \[
C_{F,P}^* = \{ \beta \in \t_\R^*:  \langle x, \beta \rangle = \psi_P(\beta), \forall x \in F\}.
\]
where $\psi_P(\beta) = \min_{y \in P} \langle y, \beta \rangle$. By \cite[Corollary A.4]{MVpolytopes}, (see Corollary \ref{corollary:A4ofMV}) $P$ is a coarsening of the Weyl fan $\mathcal{W}$ and the following corollary is immediate. 

\begin{corollary}\label{lemma:weylfancones}
For any GGMS polytope $P$ with vertex data $(\mu_\bullet)$, $C_v^* \subseteq C_{\mu_v,P}^*$ for every $v \in W$.   
\end{corollary}

\begin{definition}
Fix $w \in W$. Let $\F^w$ be the fan of $\t_\R^*$ defined by the maximal cones for $v \in W$
\[
D_v^* := \bigcup_{\substack{u \in W \\ u_w = v}} C_u^*.
\]
$D_v^*$ is indexed by $v \in W$ such that $v \leq w$. Clearly, $\F^w$ is a coarsening of the Weyl fan. 
\end{definition}

\begin{proposition} Let $w \in W$ and suppose $P$ is an MV polytope. $P \in \Pw$ if and only if $\mathcal{N}(P)$ is a coarsening of $\F^w$. 
\end{proposition}

\begin{proof}
Consider a polytope $P \in \Pw$ with vertex data $(\mu_\bullet)$ and $v \in W$ arbitrary. For every $\beta \in \t_\R^*$, $\langle \mu_v, \beta \rangle = \langle \mu_{v_w}, \beta \rangle$ so by definition $C_{\mu_v,P}^* = C_{\mu_{v_w},P}^*$. By Corollary \ref{lemma:weylfancones}, it follows that $C_v^* \subseteq C_{\mu_{v_w},P}^*$ for every $v \in W$, hence $D_v^* \subseteq C_{\mu_{v_w},P}^*$ and $\mathcal{N}(P)$ is a coarsening of $\F^w$.

For the converse, consider an MV polytope $P$ such that the dual fan is a coarsening of $\F^w$. By Corollary \ref{lemma:weylfancones}, $C_w^* \subseteq C_{\mu_w,P}^*$. As $\mathcal{N}(P)$ is a coarsening of $\F^w$, then $D_w^* \subseteq C_{\mu_w,P}^*$ as well so $C_{w_0}^* \subseteq D_w^*$ implies $C_{w_0}^* \subseteq C_{\mu_w,P}^*$. Thus for every $\beta \in C_{w_0}^*$, $\langle \mu_w , \beta \rangle = \langle \mu_{w_0}, \beta \rangle$. But this is only possible when $\mu_w = \mu_{w_0}$ so $P \in \Pw$. 
\end{proof}

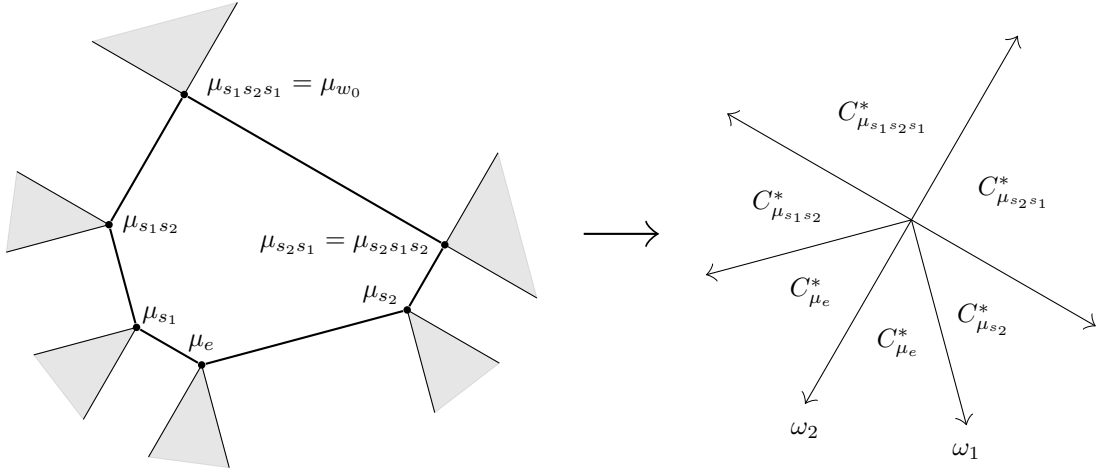
\begin{figure}
\centering
$
\begin{tikzpicture}[rotate=60]

\node[My style] (SE) at (0, 0) [label=above:$\mu_{e}$] {};
\node[My style] (S1) at (0, 1) {};
\node at (0.3,.8) {$\mu_{s_1}$};
\node[My style] (S12) at (1, 2)[label=right:$\mu_{s_1s_2}$] {};
\node[My style] (S2) at (2,-2){};
\node at (2,-1.6) {$\mu_{s_2}$};
\node[My style] (S21) at (3,-2)[label=left:{$\mu_{s_2s_1}=\mu_{s_2s_1s_2}$}] {};
\node[My style] (S121) at (3, 2) {};
\node at (3.1,2) [label=right:{$\mu_{s_1s_2s_1}=\mu_{w_0}$}]{};

\filldraw[draw=black,fill=black,opacity=0.1] (0,0) -- (-1.414,0) -- (-1, -1) -- cycle;
\filldraw[draw=black,fill=black,opacity=0.1] (3,-2) -- (4.414,-2) -- (3, -3.414) -- cycle;
\filldraw[draw=black,fill=black,opacity=0.1] (3,2) -- (3,3.414) -- (4.414, 2) -- cycle;
\filldraw[draw=black,fill=black,opacity=0.1] (0,1) -- (-1.414,1) -- (-1, 2) -- cycle;
\filldraw[draw=black,fill=black,opacity=0.1] (2,-2) -- (2, -3.414) -- (1, -3) -- cycle;
\filldraw[draw=black,fill=black,opacity=0.1] (1, 2) -- (0, 3) -- (1, 3.414) -- cycle;

\draw (0,0) -- (-1.414,0);
\draw (0,0) -- (-1, -1);

\draw (3,-2) -- (4.414,-2);
\draw (3,-2) -- (3, -3.414);

\draw (3,2) -- (3,3.414);
\draw (3,2) -- (4.414, 2);

\draw (0,1) -- (-1.414,1) ;
\draw (0,1) -- (-1,2);

\draw (2,-2) -- (2, -3.414);
\draw (2,-2) -- (1, -3);

\draw (1, 2) -- (0,3);
\draw (1,2) -- (1, 3.414);

\draw[thick] (SE) -- (S1) -- (S12) -- (S121) -- (S21) -- (S2) -- (SE);
\end{tikzpicture}
\begin{tikzpicture}

\node at (0,-3) {};
\node at (-1.5,0) {};
\node at (.5,0) {};

\draw[thick, -{{Classical TikZ Rightarrow}[scale=2]}] (-1,0) -- (0,0);
\end{tikzpicture}
\begin{tikzpicture}[rotate=60]

\node at (1, 1) {$C^*_{\mu_{s_1s_2s_1}}$};
\node at (1,-1) {$C^*_{\mu_{s_2s_1}}$};
\node at (-2/3,-3/2) {$C^*_{\mu_{s_2}}$};
\node at (-2/3,3/2) {$C^*_{\mu_{s_1s_2}}$};
\node at (-3/2, -2/3) {$C^*_{\mu_{e}}$};
\node at (-3/2, 2/3) {$C^*_{\mu_{e}}$};

\node at (-2,-2) [label=below:$\omega_1$]{};
\node at (-2.828,0) [label=below:$\omega_2$]{};

\draw[-{{Classical TikZ Rightarrow}[scale=2]}] (0,0) -- (-2.828,0);
\draw[-{{Classical TikZ Rightarrow}[scale=2]}] (0,0) -- (-2, -2);
\draw[-{{Classical TikZ Rightarrow}[scale=2]}] (0,0) -- (2.828,0);
\draw[-{{Classical TikZ Rightarrow}[scale=2]}] (0,0) -- (0, -2.828);
\draw[-{{Classical TikZ Rightarrow}[scale=2]}] (0,0) -- (0,2.828);
\draw[-{{Classical TikZ Rightarrow}[scale=2]}] (0,0) -- (-2,2);

\end{tikzpicture}
$
\caption{The dual fan of a $B_2$ polytope of highest vertex $s_1s_2s_1$}
\label{figure:dualfanB2}
\end{figure}

As a result of this correspondence, the cones of the dual fan of $P$ correspond with the vertices of $P$ while the defining rays of the maximal cones of $\mathcal{N}(P)$ will correspond with the codimension 1 faces of $P$. In the standard case, these codimension 1 faces are exactly the hyperplanes $M_\gamma$ for every $\gamma \in \Gamma$. When $w \neq w_0$, some of the hyperplanes $M_\gamma$ of $\Pw$ may have larger codimension  and hence $\Pw$ can have fewer than $|\Gamma|$ codimension 1 faces. An interesting question would be to find exactly which chamber weights label these codimension 1 faces in $\Pw$. 

\begin{question}
What are the defining rays of the maximal cones of $\F^w$?
\end{question}

These rays will correspond to some subset of the chamber weights $\Gamma_{\Pw}$. This subset will give us the defining hyperplanes of $P$, i.e. $P = \{ x: \langle x, \gamma\rangle \leq M_\gamma, \forall \gamma \in \Gamma_{\Pw}\}$. 

\begin{example}
For $B_2$ polytopes in $\P_{s_1s_2s_1}$, the hyperplanes labelled by $s_2s_1\omega_1$ and $s_1s_2s_1\omega_1$ are not defining hyperplanes of these polytopes. See Figure \ref{figure:dualfanB2} for the dual fan of such a polytope and see Figure \ref{figure:defininghyperplanesB2} for the defining hyperplanes. 
\begin{figure}[h]
\centering
$
\begin{tikzpicture}[rotate=60]

\node[My style] (SE) at (0, 0) [label=below:$\mu_{e}$] {};
\node[My style] (S1) at (0, 1) [label=left:{$\mu_{s_1}$}] {};
\node[My style] (S12) at (1, 2)[label=left:$\mu_{s_1s_2}$] {};
\node[My style] (S2) at (2,-2)[label=right:$\mu_{s_2}$] {};
\node[My style] (S21) at (3,-2)[label=right:{$\mu_{s_2s_1}=\mu_{s_2s_1s_2}$}] {};
\node[My style] (S121) at (3, 2)[label=above:{$\mu_{s_1s_2s_1}=\mu_{w_0}$}] {};

\draw[thick, -stealth] (0,1/2) -- (.707, 1/2)  node[anchor=west]{\footnotesize{$\omega_2$}};
\draw[thick, -stealth] (.5, 1.5) -- (1, 1)  node[anchor=south]{\footnotesize{$s_1\omega_1$}};
\draw[thick, -stealth] (1, -1) -- (1.5, -.5) node[anchor=south]{\footnotesize{$\omega_1$}};
\draw[thick, -stealth] (2.5, -2) -- (2.5, -1.293) node[anchor=south]{\footnotesize{$s_2\omega_2$}};
\draw[thick, -stealth] (2, 2) -- (2, 1.293) node[anchor=north]{\footnotesize{$s_1s_2\omega_2$}};
\draw[thick, -stealth] (3,0) -- (2.293,0) node[anchor=north]{\footnotesize{$s_2s_1s_2\omega_2$}};

\draw[thick] (SE) -- (S1) -- (S12) -- (S121) -- (S21) -- (S2) -- (SE);

\end{tikzpicture}
$
\caption{The hyperplanes of a $B_2$ polytope of highest vertex $s_1s_2s_1$}
\label{figure:defininghyperplanesB2}
\end{figure}
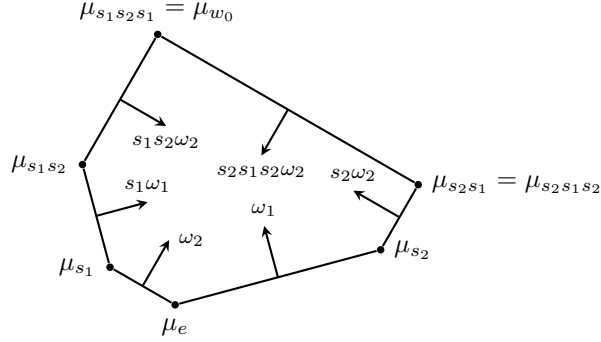
\end{example}

\begin{remark}
The normal cone of the Bruhat interval polytope $Q_{e,w}$ was described by \cite{Gaetz} in type $A$ as the equivalence class of all the linear extensions of the graphs $\Gamma_w (u)$ associated to the polytope $Q_{e,w}$. Using both descriptions of the normal cone, we expect that the set of linear extensions of the graph $\Gamma_{w}(u)$ will be exactly the set of $v \in W$ such that $v_w =u$. 
\end{remark}

\section{Tropical geometry}\label{section:tropicalpoints}

In this section, we outline the basic theory of tropical geometry to describe the correspondence between MV polytopes and the non-negative tropical points of the unipotent subgroup $N \subseteq G$. 

First, we recall the concepts of positive spaces and  tropical points as in {\cite[Section 1]{ClusterEnsembles}}.

\begin{definition} Let $X$ be an irreducible variety. A \emph{positive atlas} on $X$ is a collection of birational isomorphisms $\{ \alpha\}_{\alpha \in \mathcal{C}_X}$ over $\mathbb{Q}$ where $\alpha: T \rightarrow X$ and $T$ is a split algebraic torus. These coordinate systems satisfy:
\begin{enumerate}[label=(\roman*)]
 \item Each $\alpha$ is regular on the complement of a positive divisor in $T$ and is given by a positive rational function,
 \item For any pair $\alpha, \beta$ of coordinate systems, $\beta^{-1} \circ \alpha$ is a positive birational isomorphism of $T$.
\end{enumerate}
If $X$ has a positive atlas, we call $X$ a \em{positive space}.
\end{definition}

On an algebraic torus $T$, define the tropical points as the cocharacters of $T$, i.e $T(\Z^\trop) = X_*(T)$. Using a positive atlas, there is a unique way to define the $\Z$-tropical points of the variety $X$. 

\begin{definition}
The tropical points of a positive space $X$ is defined as 
\[
X(\Z^\trop) = \bigsqcup_\alpha T(\Z^\trop) / (\text{identifications } (\beta^{-1} \circ \alpha)^\trop).
\]
\end{definition}

For a subtraction-free function $F$ on $T$, we can tropicalize it to a function $F^\trop$ on the tropical points. To see how a tropical function is related to the original function, consider the following example:
\begin{align*}
F(x, y, z) &= \frac{xy}{z} + 2z \longmapsto F^\trop(x,y,z) = \min\{x + y -z , z\}
\end{align*}
We will call a function $F$ on $X$ \emph{positive} if it can be written as a subtraction-free expression in the coordinates of a positive atlas of $X$. We will denote the tropical function by $F^\trop$. 

\subsection{MV polytopes as tropical points}

For $G$ a reductive complex algebraic group, let $T$ be a torus of $G$, $B$ a Borel subgroup containing $T$ and $N$ its unipotent subgroup. Consider the map $x_i: \C \rightarrow N$ with image in the Chevalley subgroup of $\alpha_i$ by $x_i(a) = \exp(a E_i)$. For the tuple $\ui = (i_1, \dots, i_k)$, we define $x_{\ui}(a_1, \dots, a_k) = x_{i_1}(a_1) \cdots x_{i_k}(a_k)$. 

We will show that the variety $N$ is a positive space. For a reduced word $\ui$ of $w_0$, define the Lusztig parameterization associated to $\ui$ as the map $x_{\ui} : (\C^\times)^m \rightarrow N$ by $(a_1, \dots, a_m) \mapsto x_\ui (a_1, \dots a_m)$, where $m = \ell(w_0)$. This map is a birational isomorphism by \cite{TotalPositivityLusztig} and hence gives a coordinate system on $N$. In fact, the collection of the charts $(x_{\ui})$ form a positive atlas of $N$, called \emph{Lusztig's positive atlas} \cite{TotalPositivityLusztig}. Thus the tropical points of $N$ are defined and $N(\Z^\trop) \cong \Z^m$.

\begin{example}
Let $G = SL_3$. Since $w_0 = s_1s_2s_1 = s_2s_1s_2$, we have the two coordinates on $N$:
\begin{align*}
x_{1} (a_1) x_2(a_2) x_1(a_3) = \begin{bmatrix} 1 & a_1 + a_3 & a_1a_2 \\ 0 & 1 & a_2 \\ 0 & 0 & 1 \end{bmatrix}
&&
x_2 (b_1) x_1(b_2) x_2(b_3) = \begin{bmatrix}
 1 & b_2 & b_2 b_3 \\ 0 & 1 & b_1 + b_3 \\ 0 & 0 & 1
\end{bmatrix}.
\end{align*}
The transition maps are subtraction-free:
\[
x_{212}^{-1} \circ x_{121}(a_1, a_2, a_3) = \left(\frac{a_2a_3}{a_1 + a_3}, a_1 + a_3, \frac{a_1a_2}{a_1 + a_3}\right).
\]
Thus $N(\Z^\trop) \cong \C^3(\Z^\trop) = \Z^3$. 
\end{example}

In Section \ref{section:MVpolytopes}, we saw that the set of MV polytopes are in bijection with $\N^m$ by fixing a reduced word $\ui$ and considering the Lusztig data of $P$ with respect to $\ui$.  We would like to show that the set of MV polytopes of $G$ are in bijection with the non-negative tropical points of $N$. First, we define a positive function to pick out the ``non-negative'' points. 

Define the \emph{potential function}  $\chi: N \rightarrow \C$ by
\[
\chi(x_{\ui}(a_1, \dots, a_m)) = \sum_{k=1}^m a_k.
\]
This function is subtraction-free in the Luzstig coordinates $x_{\ui}$. In fact, $\chi$ is independent of $\ui$ and thus positive on Lusztig's positive atlas. Hence we have a tropical function $\chi^\trop$ acting on $N(\Z^\trop)$. Define the non-negative points as
\[
N(\Z^\trop)_{\geq} = \{ \ell \in N(\Z^\trop): \chi^\trop(\ell) \geq 0\}.
\]
Under the correspondence $N(\Z^\trop) \cong \Z^m$, we can write $\ell \in N(\Z^\trop)$ as $\ell = (A_1, \dots, A_m)$ for some $A_i \in \Z$. Then $\chi^\trop(\ell) \geq 0 \iff \min \{A_1, \dots, A_m\} \geq 0 \iff A_i \geq 0 \text{ for all } i$. Thus, $N(\Z^\trop)_{\geq} \cong \N^m$. By the correspondence between $\P \cong \N^m$, we can find a bijection $N(\Z^\trop)_\geq \cong \P$. 

\begin{theorem}[{\cite[Theorem 5.4]{Geometryofcanonical}}, {\cite[Theorem 4.5]{MVpolytopes}}]\label{theorem:tropicalpoints}
For $G$ a reductive algebraic group, there is a bijection between the non-negative tropical points $N(\Z^\trop)_{\geq}$ and the set of MV polytopes $\mathcal{P}$. 
\end{theorem}

As Lusztig's positive atlas consists of $x_{\ui}$ for all reduced words, this bijection is independent of the reduced word used for the Lusztig data in the  bijection $\P \rightarrow \N^m$. This bijection is also compatible with the hyperplane data in the sense that the following diagram commutes:
\[
\begin{tikzcd}
N(\Z^\trop)_{\geq} \arrow[rdd, "(\Delta_\gamma \circ \eta)^\trop", swap] \arrow[rr] && \text{MV Poytopes} \arrow[ldd, "M_\gamma"] \\
&&\\
& \Z^m & 
\end{tikzcd}
\]

As these $\Delta_\gamma \circ \eta$ functions satisfy the Pl\"{u}cker relations, a consequence of this diagram commuting is that the tropical Pl\"{u}cker relations are automatically satisfied by the tropical functions $M_\gamma$. 

For a generic MV polytope, the highest vertex is labelled by the longest element of the Weyl group, $w_0$. In this thesis, we will consider MV polytopes where the highest vertex is labelled by $w$ for some $w \in W$. We will prove that Theorem \ref{theorem:tropicalpoints} is also true for this class of polytopes, where $N$ is replaced by a subvariety of $N$. 

\subsection{Tropical geometry of reduced double Bruhat cells}

We will define functions $M_\gamma$ on the tropical points of the reduced double Bruhat cell $L^{w^{-1}}$ that will send non-negative tropical points to the BZ data associated to MV polytopes of highest vertex $w$. These functions will come from the tropicalization of the generalized minors functions. 

Recall we define the maps $x_i: \C \rightarrow N$ by $x_i(a) = \exp( a E_i)$. We similarly define $y_i: \C \rightarrow N$ by $y_i(a) = \exp(a F_i)$. For $i \in I$, we fix a representative of $s_i \in G$ by $s_i = y_i(1) x_i(-1) y_i(1)$ and thus for any $w \in W$ we can fix a representation for $w \in G$ by $w = s_{i_1} \cdots s_{i_m}$ where $\ui = (i_1, \dots, i_m)$ is a reduced word for $w$. 

\begin{definition}
For $u,v \in W$, the reduced double Bruhat cell is $L^{u,v} := N u N \cap B_- v B_-$.
\end{definition}

In particular, we are interested in the reduced double Bruhat cell $L^{w^{-1}} : = L^{e,w^{-1}} = N \cap B_- w^{-1} B_-$. Following \cite[Section 5]{Geometryofcanonical}, we have a positive structure on $L^{w^{-1}}$ described as follows. Let $x_i: {\C^\times} \rightarrow L^{w^{-1}}$ be defined as in Section \ref{section:tropicalpoints}. For the reduced word $\ui= (i_1, \dots, i_m)$ of $w^{-1}$, define  $x_\ui: (\C^\times)^m \rightarrow L^{w^{-1}}$ by $x_\ui(a_1, \dots, a_m) = x_{i_1}(a_1) \dots x_{i_m}(a_m)$. From the application of \cite[Theorem 1.2]{DoubleBruhatCellsandTotalPositivity}, this is a coordinate system on $L^{w^{-1}}$. Consider the atlas given by the charts $(x_{\ui})$ where $\ui$ runs over all reduced words of $w^{-1}$. This atlas gives a positive structure on $L^{w^{-1}}$ which we will still call \emph{Lusztig's positive atlas}. As in the case of $N$, define the potential function
\[
\chi(x_\ui(a_1, \dots, a_m) ) = \sum_{i=1}^m a_i.
\]
The potential $\chi$ is still independent of $\ui$ and is positive on this atlas, so we can define the non-negative tropical points
\[
L^{w^{-1}}(\Z^\trop)_{\geq} = \{ a\in L^{w^{-1}}(\Z^\trop) : \chi^t(a) \geq 0\}.
\]

To define the functions $M_\gamma$ on $L^{w^{-1}}(\Z^\trop)$, we need to introduce the generalized minors.

\begin{definition}\label{definition:generalizedminors} Consider the highest weight representation $V(\lambda)$ of $G$. Let $\gamma$ and $\delta$ be an extremal weights of $V(\lambda)$ and let $v_\gamma$ and $v_\delta$ be vectors in $V(\lambda)$ of weight $\lambda$ and $\delta$ respectively. Let $\langle \cdot, \cdot \rangle$ denote the Shapovalov form \cite{Shapovalovform}, i.e. $\langle F_i v, w \rangle = \langle v, E_i w\rangle$.  The \emph{generalized minors} are functions $\Delta_{\delta,\gamma}: G \rightarrow \C$ such that 
\[
\Delta_{\delta,\gamma} (g) = \langle g \cdot v_\gamma, v_\delta \rangle.
\]
We use the shorthand $\Delta_\gamma$ when $\delta = \lambda$. 
\end{definition}

Denote the subset of chamber weights  $\Gamma^w = \{ v \omega_j: j \in I, v \in W \text{ such that } v \leq_R w \} \subseteq \Gamma$. By \cite[Proposition 2.8]{ClusterAlgebrasIII}, $L^{w^{-1}}$ can be defined by the vanishing conditions of generalized minors:
\begin{align}\label{equation:reducedbruhatcell}
L^{w^{-1}} = \{ g\in N: & \Delta_{\omega_i, \omega_i}(g) =1, \Delta_{\omega_i, w\omega_i} (g)\neq 0,  \Delta_{\omega_i, v\omega_i}(g) =0 \text{ for }v\omega_i \not\in \Gamma^w\}.
\end{align}

\begin{example}\label{example:reducedbruhat1}
Let $G = SL_3$. The fundamental weights can be realized as $\omega_1 = (1,0,0)$ and $\omega_2 = (1,1,0)$.  We use the shorthand $\omega_1 = 1$ and $\omega_2 = 12$, where each number indicates which coordinate is equal to $1$. The simple reflections act as the transposition $s_1 = (12)$ and $s_2 = (23)$ on the fundamental weights. 

When $w = s_1s_2$, the reduced Bruhat cell is given by
\[
L^{s_2s_1} = \left\lbrace x_2(\beta) x_1(\alpha) = \begin{bmatrix} 1 & \alpha & 0 \\ 0 & 1 & \beta \\ 0 & 0 & 1 \end{bmatrix} : \alpha, \beta \in \C^\times \right\rbrace. 
\] 
Note that $\Gamma^{s_1s_2} =\{ 1, 2, 12, 13, 23 \}$. Indeed, $\Delta_1 = 1$, $\Delta_{12}=1$, $\Delta_2 = \alpha$, $\Delta_{23} = \alpha\beta$, $\Delta_{13} = \beta$ are all nonzero but $\Delta_3=0$ as $ 3 \not\in \Gamma^{s_1s_2}$. 
\end{example}

Define the map $\eta_{w^{-1}}$ on $L^{w^{-1}}$ by setting $\eta_{w^{-1}}(x)$ to be the unique element in $N \cap B_- \overline{w^{-1}} x^T$. By \cite[Theorem 1.2, Proposition 1.3]{TotalPositivityin}, $\eta_{w^{-1}}$ is a regular automorphism of $L^{ w^{-1}}$ and $\eta_{w^{-1}}^{-1}(z) = (\eta_{w}(z^\iota))^\iota$, where $\iota$ is the anti-automorphism that sends $x_{i_1}(t_1) \dots x_{i_m}(t_m) \mapsto x_{i_m}(t_m) \dots x_{i_1}(t_1)$.

Define the $y$-coordinates $y_{\ui}(b_\bullet) := \eta_{w^{-1}}^{-1}(x_\ui(b_\bullet)) = \iota \circ \eta_w (x_{\ui}(b_\bullet)^\iota)$; these are the coordinates used in the proof of {\cite[Theorem 7.1]{MVpolytopes}} (see Theorem \ref{theorem:MVpolytopeslusztigdata}). For $\gamma \in \Gamma^w$, define $M_\gamma  = (\Delta_\gamma \circ \eta_{w^{-1}}^{-1})^\trop$. Then $(M_\gamma)_{\gamma \in \Gamma^w}$ is a collection of functions on the tropical points of $L^{w^{-1}}$. 

\begin{example}\label{example:reducedbruhat2}
Continuing Example \ref{example:reducedbruhat1}, the $y$-coordinates on $L^{s_2s_1}$ are given by \[
y_{21}(\alpha^{-1}, \beta^{-1}) = \eta_{w^{-1}}^{-1}(x_{2}(\beta) x_1(\alpha)) = \begin{bmatrix}
 1 & \alpha^{-1} & 0 \\ 0 & 1 & \beta^{-1} \\ 0 & 0 & 1 \end{bmatrix}.
\]
For $(a,b) \in L^{s_2s_1}(\Z^\trop)$, the functions $M_\gamma$ take on the values
\begin{align*}
M_1 (a,b) = 0, && M_{12}(a,b) =0, && M_2(a,b) = -a, && M_{23}(a,b) = -a-b, && M_{13}(a,b) = -b.
\end{align*}
\end{example}

Consider $\gamma \not\in \Gamma^w$. By (\ref{equation:reducedbruhatcell}), $\Delta_\gamma=0$.  We would like to  redefine these generalized minors so that we have functions which are non-zero on $L^{w^{-1}}$. 
\begin{definition}
For $v \in W$ and $s_i \in D_L(v)$, define $\Delta_{v\omega_i}^\new := \Delta_{v_w^{-1} v \omega_i, v \omega_i}$. 
\end{definition}

\begin{example}\label{example:reducedbruhat3}
Continuing \ref{example:reducedbruhat2}, we redefine the minor $\Delta^\new_3= \Delta_{s_1\omega_1, w_0 \omega_1} = \Delta_{2,3} = \beta$. Note that this is the smallest row set which results is a nonzero minor with the column set. 
\end{example}

For $\gamma \in \Gamma$, define $M_\gamma = (\Delta_\gamma^\new \circ \eta_{w^{-1}}^{-1})^\trop$. Note that $M_\gamma = (\Delta_\gamma \circ \eta_{w^{-1}}^{-1})^\trop$ for $\gamma \in \Gamma^w$. We will show that for each $\ell \in L^{w^{-1}}(\Z^\trop)_\geq$, $(M_\gamma(\ell))_{\gamma\in\Gamma}$ are the BZ data of some $P \in \Pw$. To do this, we need to show that $(M_\gamma(\ell))_{\gamma \in \Gamma}$ is a BZ datum and that the edge equalities in \ref{condition:vanishing} of Lemma \ref{lemma:BZdatahighestvertexw} hold.

First, we will show that this $\Delta_\gamma^\text{new}$ is the ``smallest'' non-zero minor; this will imply the edge equalities. We start with a technical lemma.  Recall the partial ordering on $X^*$ where $a\leq b \iff b-a \in Q_+$. 

\begin{lemma}\label{lemma:claimforzerominor}
For $b \in B$, $\alpha \in X^*$, $u \in W$, and $\lambda$ a dominant weight, if $\langle v_\alpha, buv_{\lambda} \rangle \neq 0$, then $\alpha  = u \lambda + \beta$ for $\beta \in \Delta_+ \cap u\Delta_-$. 
\end{lemma}

\begin{proof}
Consider the Lie algebra of $G$, $\g$.  Recall the root space decomposition of $\g = \h \bigoplus_{\alpha \in \Delta} \g_\alpha$, where $\h$ is the Cartan subalgebra and $\g_\alpha = \{ x \in \g: [h,x] = \alpha(h)x, \, \forall h \in \h \}$. Set $\b = \h \bigoplus_{\alpha \in \Delta_+} \g_\alpha$. 

For $\lambda$ a dominant weight, consider the representation $V(\lambda)$. The Demazure module is defined as $V_u(\lambda) = U(\b) \cdot v_{u\lambda}$, where $v_{u\lambda}$ is a vector of weight $u\lambda$ in the 1-dimensional $u\lambda$-weight space of $V(\lambda)$. We will show that the weights of $V_u(\lambda)$ is the set $u\lambda + \Delta_+ \cap u\Delta_-$. 

First, consider $\n_u = \bigoplus_{\alpha \in \Delta_+ \cap u \Delta_-} \g_\alpha$ and $\n_{-u} = \bigoplus_{\alpha \in \Delta_+ \cap u \Delta_+} \g_\alpha$. Then $\b = \n_u \oplus \n_{-u} \oplus \h$ so we have a PBW basis $ABC$, where $A$ is a product of vectors from $\n_u$, $B$ is a product of vectors in $\n_{-u}$ and $C$ is a product of vectors from $\h$. 

Suppose that $\beta \in \Delta_+ \cap u\Delta_+$. Then $u^{-1} \beta \in \Delta_+$ as well. Consider $X_\beta \in B$. Then
\begin{align*}
X_\beta \cdot v_{u\lambda} & = X_{\beta} u \cdot v_{\lambda} = u (u^{-1} X_{\beta} u) v_{\lambda}.
\end{align*}
Note that $u^{-1} X_{\beta} u = (\text{ad}_u(X_\beta)) \cdot v = X_{u^{-1} \beta}$. But as $u^{-1} (\beta) \in \Delta_+$ and $v_\lambda$ is a highest weight vector of $V(\lambda)$, then $X_{u^{-1} \beta} \cdot v_{\lambda}=0$. 

Thus for every $b\in B$ with $b \neq 1$, $b \cdot v_{u\lambda}=0$. As $C$ does not affect the weight of $v_{u\lambda}$, then $BC v_{u\lambda}$ only has weights $u\lambda$. Thus $U(\b) \cdot v_{u\lambda} = ABC \cdot v_{u\lambda}$ has weights $u\lambda - \Delta_+ \cap u\Delta_-$ as desired. 
\end{proof}

\begin{lemma}\label{lemma:zerocondition}
Let $\lambda$ be a dominant weight and let $g \in L^{w^{-1}}$. Then $\Delta_{u\lambda, \mu}(g) = 0$ for any $\mu \not\geq wu\lambda$.
\end{lemma}

\begin{proof}
Take $g \in L^{w^{-1}}$, then $g = b_1 w^{-1} b_2$ for $b_1, b_2 \in B_-$. Let $\mu\in X^*$ be a weight of $V(\lambda)$ so $\Delta_{u\lambda, \mu}(g) = \langle g v_\mu, uv_{\lambda} \rangle$ for vectors $v_\lambda, v_{\mu} \in V(\lambda)$ of weights $\lambda$ and $\mu$ respectively. By the definition of the Shapovalov form, $\langle g v_\mu, u v_\lambda \rangle = \langle w^{-1} b_2 v_\mu, b_2^t uv_\lambda \rangle =  \langle v_\mu, g^t uv_\lambda \rangle$, where $g^t = b_2^t w b_1^t$. 

 Note that $g^t uv_\lambda  = \sum_{\alpha \in Q} \langle v_\alpha, g^t u v_\lambda \rangle v_\alpha $ for weight vectors $v_\alpha \in V(\lambda)$ of weight $\alpha$. In fact,
  \[
 g^t u v_\lambda =  \sum_{\alpha\in X^*} \langle w^{-1} b_2 v_{\alpha}, b_1^t u v_\lambda \rangle v_\alpha.
 \]
Let $w^{-1}b_2 v_\alpha = \sum_{\eta \in X^*} \langle w^{-1}b_2 v_\alpha, v_\eta \rangle v_\eta$. Then 
  \[
 g^t u v_\lambda =  \sum_{\alpha, \eta \in X^*}\langle w^{-1}b_2 v_\alpha, v_\eta \rangle \langle  v_\eta, b_1^t u v_\lambda \rangle v_\alpha.
 \] 
 By Lemma \ref{lemma:claimforzerominor}, $\langle  v_\eta, b_1^t u v_\lambda \rangle \neq 0 \iff \eta = u\lambda + \beta$ for $\beta \in \Delta_+ \cap u \Delta_-$. Since $B_-$ always lowers the weights,  $\langle w^{-1} b_2 v_\alpha, v_\eta \rangle \neq 0 \iff w\eta = \alpha - \gamma$ for some $\gamma \in Q_+$. Thus
 \[
 \alpha = wu\lambda + w \beta + \gamma
 \]
By \cite[Corollary 2.3]{TotalPositivityin}, as $wu$ is reduced then $\Delta_+ \cap u^{-1}(\Delta_-) \subseteq \Delta_+ \cap (wu)^{-1} (\Delta_-)$. Note that $\alpha \in \Delta_+ \cap xy(\Delta_-) \iff -x^{-1}\beta \in y\Delta_+ \cap x^{-1}(\Delta_-)$. Thus we also have the inclusion
\[
\Delta_+ \cap u(\Delta_-) \subseteq w^{-1} \Delta_+ \cap u(\Delta_-).
\]
Then $\beta \in w^{-1} \Delta_+ \cap u (\Delta_-)$ as well so there exists $\delta \in \Delta_+$ such that $\beta = w^{-1} \delta \iff w\beta = \delta$. Thus $w\beta \in \Delta_+$ and $\alpha = wu\lambda + \delta + \gamma \in wu\lambda + Q_+$.  Hence 
 \[
 g^t u v_\lambda = \sum_{\gamma \in  Q_+} \langle v_{wu\lambda + \gamma}, g^t u v_\lambda \rangle v_{wu\lambda + \gamma}
 \]
 and so $\Delta_{u\lambda, \mu}(g) = 0$ for $\mu \not\geq wu\lambda$. 
\end{proof}

\begin{corollary}\label{lemma:zerominorwu}
Fix $w \in W$ and let $u \leq_R w^{-1}w_0$. For every $s_i \in S$ such that $wus_i$ is a reduced product, then $\Delta_{u\omega_i, wus_i\omega_i} =0$ on $L^{w^{-1}}$. 
\end{corollary}

\begin{proof}
Since $wus_i$ is reduced, then $wu\alpha_i \in \Delta_+$. As $s_i \omega_i = \omega_i - \alpha_i$, then $wu s_i \omega_i = wu\omega_i - wu\alpha_i$ and hence $wus_i\omega_i \leq wu\omega_i$. By setting $\lambda = \omega_i$  and $\mu = wus_i\omega_i$, we can apply Lemma \ref{lemma:zerocondition} to see that $\Delta_{u\omega_i, wus_i\omega_i}(g)=0$ for $g \in L^{w^{-1}}$.
\end{proof}

\begin{conjecture}\label{lemma:zerominorvw}
Fix $w \in W$, let $v \in W$. Set $u = v_w^{-1}v$ and let $s_i \not\in D_R(v)$ such that $(vs_i)_w = v_w$. Then $\Delta_{u\omega_i, vs_i\omega_i} =0$ on $L^{w^{-1}}$. 
\end{conjecture}

\begin{remark}
This conjecture is known for a few special cases. When $v_w = w$, then $u= w^{-1}v$ is an initial word of $w^{-1}w_0$ and hence the conjecture is equivalent to Corollary \ref{lemma:zerominorwu}. On the other hand, when $v_w = v$ then $u = e$ so the generalized minor of interest is of the form $\Delta_{\omega_i, vs_i\omega_i}$. By assumption, $v \leq w$ but $s_i$ is such that $(vs_i)_w = v$ and hence $vs_i \not \leq w$ by maximality of $(vs_i)_w$. Thus the conjecture follows from \cite[Proposition 2.8]{ClusterAlgebrasIII} (see (\ref{equation:reducedbruhatcell})).
\end{remark}

These two results imply the edge equalities are satisfied for large enough $\gamma$. 

\begin{proposition}\label{proposition:edgeequalities}
$(M_\gamma)_{\gamma \in \Gamma}$ satisfy the edge equalities \ref{condition:vanishing} of Lemma \ref{lemma:BZdatahighestvertexw}. In other words, for every $v\in W$ and $s_i \not\in D_R(v)$ such that $\mu_{vs_i} = \mu_v$, 
\[
  M_{v\omega_i} + M_{vs_i\omega_i} = - \sum_{j \neq i} a_{j,i} M_{v\omega_j}. 
\]
\end{proposition}

\begin{proof}
By \cite[Proposition 4.1]{Tensorproductmultiplicities}, for every $u, w$ such that $\ell(us_i) = \ell(u) +1$, $\ell(wus_i) = \ell(w) + \ell(u) + 1$, 
\[
\Delta_{u \omega_i, w u\omega_i} \Delta_{us_i \omega_i, wus_i \omega_i} = \Delta_{u s_i \omega_i, w u \omega_i} \Delta_{u \omega_i, w us_i \omega_i} + \prod_{j \neq i} \Delta_{u \omega_j, wu \omega_j}^{-a_{j,i}}
\]
By Corollary \ref{lemma:zerominorwu}, $\Delta_{u\omega_i, wus_i\omega_i}=0$ on $L^{w^{-1}}$ and by tropicalizing, we obtain
\[
M_{w u\omega_i} + M_{wus_i \omega_i} = \sum_{j \neq i}(-a_{j,i}) M_{wu \omega_j}
\]
which are exactly the edge equalities. 

For $v \in W$ arbitrary, set $u = v_w^{-1}v$. For $s_i$ such that $\ell(vs_i) = \ell(v) +1$ and $(vs_i)_w = v_w$, then by \cite[Proposition 4.1]{Tensorproductmultiplicities}
\[
\Delta^\text{new}_{v\omega_i} \Delta^\text{new}_{v s_i \omega_i} = \Delta_{u s_i \omega_i, v \omega_i} \Delta_{u \omega_i, v s_i \omega_i} + \prod_{j \neq i} \left(\Delta^\text{new}_{v \omega_j}\right)^{-a_{j,i}}
\]
By Conjecture \ref{lemma:zerominorvw}, $\Delta_{u\omega_i, vs_i\omega_i} =0$ on $L^{w^{-1}}$ and hence by tropicalizing, 
\[
M_{v\omega_i} + M_{vs_i\omega_i} = \sum_{j \neq i} (-a_{ji}) M_{v\omega_j}.
\]
So we have proved that for every $v \in V$ such that $\mu_{vs_i} = \mu_v = \mu_{v_w}$, the edge equalities $M_{v\omega_i} + M_{vs_i\omega_i} = \sum_{j \neq i} (-a_{ji}) M_{v\omega_j}$ are satisfied.  
\end{proof}

\begin{theorem}\label{theorem:BZdata}
There is a bijection $L^{w^{-1}}(\Z^\trop)_{\geq} \longrightarrow \Pw$ by 
\[
\ell \rightarrow (M_\gamma(\ell))_{\gamma \in \Gamma}.
\]
\end{theorem}

\begin{proof}
First, we show is that the collection $(M_\gamma)_{\gamma\in \Gamma}$ is the BZ datum of an MV polytope in $\Pw$. The collection $(M_{\gamma})_{\gamma \in \Gamma^w}$ satisfies the tropical Pl\"{u}cker relations as $\Delta_\gamma$ satisfies the Pl\"{u}cker relations. The collection $(M_\gamma)_{\gamma \in \Gamma}$ satisfies the edge equalities \ref{condition:vanishing} of Lemma \ref{lemma:BZdatahighestvertexw} by Proposition \ref{proposition:edgeequalities} and thus we can recursively define these tropical functions by the collection $(M_\gamma)_{\gamma \in \Gamma^w}$ using the relation (\ref{equation:vanishing}). It is easy to see that these $(M_{\gamma})_{\gamma \in \Gamma\backslash \Gamma^w}$ will also satisfy the tropical Pl\"{u}cker relations, thus $(M_\gamma)_{\gamma \in \Gamma}$ is the BZ datum of some MV polytope, $P \in \P$. Finally, by Lemma \ref{lemma:BZdatahighestvertexw}, $P \in \Pw$ and so this map is well defined.

To show this map is a bijection, fix a reduced word $\ui = (i_1, \dots, i_m)$ of $w_0$ such that a $(i_1, \dots, i_{\ell(w)})$ is a reduced word for $w$. The map 
\begin{align}\label{equation:bijectionfromBZtolusztig}
(M_\gamma(\ell))_{ \gamma \in \Gamma} \mapsto \left(-M_{w_k^\ui s_{k+1} \omega_{k+1}} - M_{ w_k^\ui \omega_{k+1}} + \sum_{j \neq i_{k+1}} a_{j, i_{k+1}} M_{w_{k+1}^\ui \omega_j}\right)_{k=0}^{m-1}
\end{align}
sends the BZ data of $P$ to the Lusztig data of $P$ with respect to the reduced word of $\ui$. By Proposition \ref{proposition:zeros}, $n_{k} =0$ for $k > \ell(w)$ so by Theorem \ref{theorem:MVpolytopeslusztigdata}, (\ref{equation:bijectionfromBZtolusztig}) is a bijection from $\Pw \rightarrow \N^{\ell(w)}$.  But $L^{w^{-1}}(\Z^\trop)_\geq \cong \N^{\ell(w)}$ so by composing these maps, the inverse $(M_\gamma(\ell))_{\gamma \in \Gamma} \mapsto \ell$ is a bijection. 
\end{proof}

\bibliographystyle{alpha}
\bibliography{bibliography.bib}

\end{document}